\documentclass{article}
\usepackage[utf8]{inputenc}
\usepackage[english]{babel}
\usepackage{hyperref}
\usepackage{amsmath,amssymb,amsthm}
\usepackage{thmtools}
\usepackage{mathtools}
\usepackage{tikz-cd}
\usepackage{csquotes}
\usepackage[shortlabels]{enumitem}
\usepackage[capitalize]{cleveref}
\usepackage{natbib}

\newtheorem{theorem}{Theorem}[section]
\newtheorem{lemma}[theorem]{Lemma}
\newtheorem{proposition}[theorem]{Proposition}
\newtheorem{corollary}[theorem]{Corollary}
\theoremstyle{definition}
\newtheorem{definition}[theorem]{Definition}
\newtheorem{assumption}[theorem]{Assumption}
\theoremstyle{remark}
\newtheorem{remark}[theorem]{Remark}
\newtheorem{examples}[theorem]{Examples}

\newcommand{\N}{\mathbb{N}}
\newcommand{\Z}{\mathbb{Z}}
\newcommand{\R}{\mathbb{R}}
\newcommand{\C}{\mathbb{C}}
\DeclareMathOperator\supp{supp}
\newcommand{\hatstar}{\mathbin{\hat{\star}}}
\newcommand{\checkstar}{\mathbin{\check{\star}}}

\newcommand{\supportGraphic}{\includegraphics{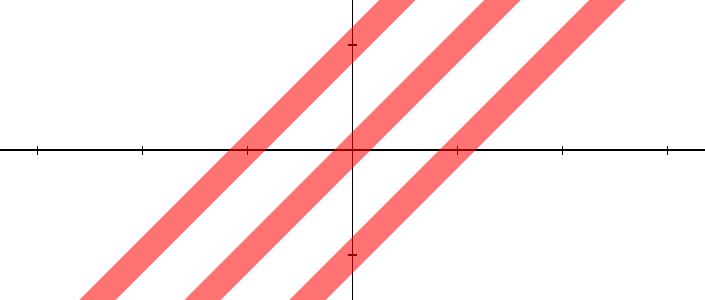}}

\begin{document}

\title{
  Group Cross-Correlations with Faintly Constrained Filters
}

\author{Benedikt Fluhr}

\maketitle


\begin{abstract}
  Group convolutional layers
  with respect to some group $G$
  are modeled by convolutions
  or cross-correlations with a filter,
  and they provide the fundamental building block
  for group convolutional neural networks.
  For entirely unconstrained filters
  and $G$ a non-abelian group,
  any hidden layer of such a network
  requires as many nodes as vertices
  in a fine enough discretization of $G$.
  In order to reduce the necessary number of nodes,
  certain constraints on filters
  were proposed in the literature.
  We propose weaker constraints
  retaining this benefit
  while also resolving an incompatibility
  previous constraints have for group actions with non-compact stabilizers.
  Moreover,
  we generalize
  previous results to group actions
  that are not necessarily transitive,
  and we weaken the common assumption that $G$ is unimodular.
\end{abstract}

\section{Introduction}
\label{sec:intro}

Group convolutional neural networks
with respect to some group $G$
were introduced by \citet{pmlr-v48-cohenc16}.
The distinguishing feature of such a neural network
is that some of its layers
are modeled by cross-correlations
of functions
with some \emph{filter} henceforth denoted by $\omega$ (or $\omega'$).
Initially,
\citet{pmlr-v48-cohenc16}
impose no constraints on the filter $\omega$
(there denoted by $\psi$)
and as a result,
within any hidden layer feature vectors
are modeled by functions defined on the entire group $G$
rather than a quotient of $G$
when $G$ is non-abelian.
\citet[Definition 5]{pmlr-v80-kondor18a} resolved this limitation
by imposing a constraint on $\omega$
(there denoted by $\chi$),
which can be paraphrased as
\emph{bi-invariance}.
\citet{10.5555/3454287.3455107}
generalized cross-correlations
from transformations of vector-valued functions
to transformations of sections
to $G$-equivariant vector bundles.
In order to accommodate the $G$-action on total spaces,
the bi-invariance constraint
of \mbox{\citep[Definition 5]{pmlr-v80-kondor18a}}
is generalized to their constraint
of \emph{bi-equivariance}.
As we will see in \cref{sec:comparison},
this constraint of bi-equivariance (or bi-invariance) is too strict,
when stabilizers are not compact.
In the present paper we propose a weaker constraint on $\omega$
with \eqref{eq:omegaConstraint}
that also works for non-compact stabilizers.
Thinking of bi-equivariance as having both left-equivariance
and right-equivariance,
our constraint can be paraphrased as
\enquote{equivariance with respect to conjugation},
which is implied by bi-equivariance.



Other than the compactness of stabilizers,
a common assumption is that the considered $G$-action
should be transitive.
In our \cref{definition:xcorr} of a cross-correlation,
transitivity is not imposed.
We also weaken the assumption
that $G$ is unimodular.
Otherwise,
we adopt the setting and the generality
by \citet{10.5555/3454287.3455107}.
In order to accommodate non-transitive $G$-actions
when comparing cross-correlations to integral transforms,
we introduce \emph{orbitwise integral transforms}
in \cref{sec:orbitwiseIntegralTrafos}
and we recover the widely known correspondence
between $G$-equivariant integral transforms
and $G$-equivariant kernels
henceforth denoted by $\kappa$.
Finally,
we establish a close relation
between $G$-equivariant orbitwise integral transforms
and cross-correlations
in \cref{sec:integralTrafosAsXCorr}.

\section{Group Cross-Correlations}
\label{sec:xcorr}

We start with a concrete example which will then be generalized
so as to obtain a rather broad notion of a \emph{cross-correlation}
with \cref{definition:xcorr}.
As a group 
we consider the real numbers $\R$ with addition
and as an $\R$-space we consider 
the unit circle ${S^1 \subset \C}$
endowed with the $\R$-action
\begin{equation}
  \label{eq:actionOnCircle}
  \_.\_ \colon
  \R \times S^1 \rightarrow S^1, \,
  (t, p) \mapsto t.p \coloneqq e^{ti} p.
\end{equation}
Now suppose we meant to design a neural network layer
having continuous functions ${S^1 \rightarrow \R}$
as inputs and as outputs.
As inputs such functions
could for example describe temperatures measured in each point
or velocities in counter-clockwise direction
of a fluid constrained to the circle.
Moreover,
suppose that the receptive field of each point
should be limited to a small neighborhood.
Such a transformation of a function
${f \colon S^1 \rightarrow \R}$
can be obtained as a \emph{cross-correlation}
with a \emph{filter}
${\omega' \colon \R \rightarrow \R}$
supported on a small neighborhood of ${0 \in \R}$:
\begin{equation}
  \label{eq:xcorrHat}
  \omega' \hatstar f \colon
  S^1 \rightarrow \R,\,
  p \mapsto
  \int_{-\infty}^{\infty} \omega'(t) f(t.p) dt.
\end{equation}
Here we added the hat to the $\star$-operator
as we will use the unaltered $\star$-operator
to denote the general form of cross-correlations
provided by \cref{definition:xcorr}.

The action \eqref{eq:actionOnCircle}
also induces an $\R$-action on functions
\begin{equation*}
  \_.\_ \colon \R \times \R^{S^1} \rightarrow \R^{S^1}, \,
  (t, f) \mapsto t.f
\end{equation*}
where $t.f$ is the rotation of $f$ (as its graph)
by $t$ in counter-clockwise direction,
i.e. we have
\begin{equation}
  \label{eq:actionOnFunctionsOnCircle}
  (t.f)(p) \coloneqq f\big(e^{-ti}p\big) = f((-t).p)
\end{equation}
for all ${t \in \R}$,
${f \colon S^1 \rightarrow \R}$,
and ${p \in S^1}$.
Note that in the present paper
function application takes precedence over the binary
operator
\enquote{$\_.\_$} denoting group actions.
Moreover,
the operation
\begin{equation*}
  \omega' \hatstar - \,\colon 
  f \mapsto \omega' \hatstar f
\end{equation*}
is $\R$-equivariant,
i.e. for any ${t \in \R}$,
continuous ${f \colon S^1 \rightarrow \R}$,
and ${p \in S^1}$ we have the equation
\begin{equation}
  \label{eq:equivarianceCircle}
  \begin{split}
    \big(\omega' \hatstar t.f\big)(p)
    &\stackrel{\eqref{eq:xcorrHat}}{=}
      \int_{-\infty}^{\infty}
      \omega'(s) (t.f)(s.p) ds
    \\
    &\stackrel{\eqref{eq:actionOnFunctionsOnCircle}}{=}
      \int_{-\infty}^{\infty}
      \omega'(s) f((-t).s.p) ds
    \\
    &\,=
      \int_{-\infty}^{\infty}
      \omega'(s) f((s-t).p) ds
    \\
    &\,=
      \int_{-\infty}^{\infty}
      \omega'(s) f(s.(-t).p) ds
    \\
    &\stackrel{\eqref{eq:xcorrHat}}{=}
      \big(\omega' \hatstar f\big)((-t).p)
    \\
    &\stackrel{\eqref{eq:actionOnFunctionsOnCircle}}{=}
      \big(t.\big(\omega' \hatstar f\big)\big)(p)
  \end{split}
\end{equation}
or equivalently
\begin{equation}
  \label{eq:equivarianceCircleMod}
  \begin{split}
    \big(\omega' \hatstar t.f\big)(t.p)
    &\stackrel{\eqref{eq:equivarianceCircle}}{=}
      \big(t.\big(\omega' \hatstar f\big)\big)(t.p)
    \\
    &\stackrel{\eqref{eq:actionOnFunctionsOnCircle}}{=}
      \big(\omega' \hatstar f\big)((-t).t.p)
    \\
    &\,=
      \big(\omega' \hatstar f\big)((t-t).p)
    \\
    &\,=
      \big(\omega' \hatstar f\big)(p).
  \end{split}
\end{equation}

\subsection{Generalization to Non-Abelian Groups}
\label{sec:toNonAbelian}

Now let $G$ be a Hausdorff topological group
with $e$ its neutral element
and let $B$ be a Hausdorff $G$-space throughout,
i.e. we have a continuous $G$-action
\begin{equation*}
  \_.\_ \colon
  G \times B \rightarrow B, \,
  (g, b) \mapsto g.b.
\end{equation*}
(We use the letter \enquote{$B$} as it will also serve as a base space
for vector bundles from \cref{sec:mackey} onward.)
We now consider the not necessarily abelian group $G$
acting on $B$
in place of the additive group $\R$
acting on $S^1$
in order to see where we get stuck,
when we try to obtain a counterpart
to equation \eqref{eq:equivarianceCircleMod}
or equivalently $G$-equivariance.

In analogy to the above example
we assume we have
a continuous compactly supported function
${\omega' \colon G \rightarrow \R}$
and we also adapt the
$\hatstar$-operator from \eqref{eq:xcorrHat} by setting
\begin{equation}
  \label{eq:xcorrHatGeneralized}
  \big(\omega' \hatstar f\big)(b) \coloneqq
  \int_G \omega'(h)f(h.b)d\mu'(h)
\end{equation}
for all continuous ${f \colon B \rightarrow \R}$
and ${b \in B}$,
where
${\mu' \colon \mathcal{B}(B) \rightarrow [0, \infty]}$
is some locally finite Borel measure on $B$.
We now examine where the operation
\begin{equation*}
  \omega' \hatstar - \,\colon 
  f \mapsto \omega' \hatstar f
\end{equation*}
fails to be $G$-equivariant.
The $G$-action on functions is defined completely analogously as
\begin{equation*}
  \_.\_ \colon G \times \R^B \rightarrow \R^B, \,
  (g, f) \mapsto g.f
\end{equation*}
where
\begin{equation*}
  g.f \colon B \rightarrow \R,\,
  b \mapsto (g.f)(b) \coloneqq f\big(g^{-1}.b\big).
\end{equation*}

Now
let ${g \in G}$, ${f \colon B \rightarrow \R}$,
and ${b \in B}$.
Then a counterpart to the left-hand side of \eqref{eq:equivarianceCircleMod}
simplifies as
\begin{equation}
  \label{eq:firstAttemptXCorrLeft}
  \begin{split}
    \big(\omega' \hatstar g.f\big)(g.b)
    &=
      \int_G \omega'(h)(g.f)(hg.b)d\mu'(h)
    \\
    &=
      \int_G \omega'(h)f\big(g^{-1}hg.b\big)d\mu'(h).    
  \end{split}
\end{equation}
If $G$ was abelian
(or if $g$ was in the center of $G$),
then we could use ${g^{-1}hg = h}$
to continue with a simplification similar to \eqref{eq:equivarianceCircle}
ending in \eqref{eq:xcorrHatGeneralized}.
In order to get rid of the conjugation by $g^{-1}$
in the last term of \eqref{eq:firstAttemptXCorrLeft},
we make both
the measure for integration $\mu'$
and the filter $\omega'$ dependent on the argument
provided to cross-correlations,
here $b$ in \eqref{eq:xcorrHatGeneralized}
and $g.b$ in \eqref{eq:firstAttemptXCorrLeft}.

\sloppy
So for the remainder of this paper excluding appendices
we assume we have a family
\begin{equation*}
  \{\mu_b \colon \mathcal{B}(G) \rightarrow [0, \infty]\}_{b \in B}
\end{equation*}
of locally finite Borel measures
that is compatible with the group action ${G \curvearrowright B}$
in the sense that
\begin{equation}
  \label{eq:muInv}
  \mu_{g.b} = c_{g*} \mu_b
\end{equation}
for any ${g \in G}$ and ${b \in B}$,
where
${c_{g*} \mu_b}$ is the pushforward measure
of $\mu_b$ along the conjugation
${c_g \colon G \rightarrow G,\, h \mapsto ghg^{-1}}$.
For now, we also assume we have a continuous function
\begin{equation*}
  \omega \colon G \times B \rightarrow \R
\end{equation*}
so as to provide a filter
${\omega(-, b) \colon G \rightarrow \R}$
for any ${b \in B}$
subject to the constraint
\begin{equation}
  \label{eq:omegaConstraintSimple}
  \omega\big(ghg^{-1}, g.b\big) = \omega(h, b)
\end{equation}
for all ${g, h \in G}$ and ${b \in B}$.
Moreover,
in order to obtain well-defined cross-correlations,
we assume each of the partially applied (filter) functions
${\omega(-, b) \colon G \rightarrow \R}$
for ${b \in B}$
to be compactly supported.
Then for a continuous function
${f \colon B \rightarrow \R}$,
we define a cross-correlation by
\begin{equation}
  \label{eq:xcorrCheck}
  \omega \checkstar f \colon
  B \rightarrow \R, \,
  b \mapsto
  \int_G \omega(h, b) f(h.b) d\mu_b(h).
\end{equation}
Indeed,
as a counterpart to \eqref{eq:equivarianceCircleMod}
we obtain the equation
\begin{equation}
  \label{eq:equivarianceFunctionsMod}
  \begin{split}
    (\omega \checkstar g.f)(g.b)
    &\stackrel{\eqref{eq:xcorrCheck}}{=}
      \int_G \omega(h, g.b) (g.f)(hg.b) d\mu_{g.b}(h)
    \\
    &\stackrel{\eqref{eq:muInv}}{=}
      \int_G \omega(h, g.b) f\big(g^{-1}hg.b\big) dc_{g*}\mu_{b}(h)
    \\
    &\,=
      \int_G \omega\big(ghg^{-1}, g.b\big) f(h.b) d\mu_{b}(h)
    \\
    &\stackrel{\eqref{eq:omegaConstraintSimple}}{=}
      \int_G \omega(h, b) f(h.b) d\mu_{b}(h)
    \\
    &\stackrel{\eqref{eq:xcorrCheck}}{=}
      (\omega \checkstar f)(b)    
  \end{split}
\end{equation}
for all ${g \in G}$ and ${b \in B}$.
Here the third equality follows from the change of variables formula
for the pushforward measure $c_{g*}\mu_b$.

\subsubsection{Continuity of Cross-Correlations}

\sloppy
In order for the function defined in \eqref{eq:xcorrCheck}
to be continuous,
we also assume the family $\{\mu_b\}_{b \in B}$
to be \emph{continuous}
in the sense that
\begin{equation*}
  B \rightarrow \R, \,
  b \mapsto \int_G f'(h) d \mu_b(h)
\end{equation*}
is a continuous function
for any \emph{compactly supported} continuous
${f' \colon G \rightarrow \R}$.
Here the \emph{support} of $f'$ --
denoted by ${\supp f'}$ --
refers to the closure of the set
${\{g \in G \mid f'(g) \neq 0\}}$
and by saying that $f'$ is \emph{compactly supported}
we mean that its support is compact.
In the Appendix \ref{sec:famHaarMeas}
and specifically \cref{thm:existenceFamHaarMeas}
we show how such a family of measures
can be constructed in a natural way.

\subsubsection{Constraints as Stabilizer Invariances}
\label{sec:contraintsInvariances}

Even though the filter $\omega$ now depends on a point ${b \in B}$
as a second argument,
the constraint \eqref{eq:omegaConstraintSimple} entails
that the partially applied function
${\omega(-, g.b) \colon G \rightarrow \R}$
is fully determined by
${\omega(-, b) \colon G \rightarrow \R}$
for any ${g \in G}$ and ${b \in B}$.
In particular,
if $G$ acts transitively on $B$,
then providing a filter as ${\omega \colon G \times B \rightarrow \R}$
is the same as to provide just one of the partially applied functions
${\omega(-, b) \colon G \rightarrow \R}$ for some ${b \in B}$
satisfying the constraint
\begin{equation}
  \label{eq:omegaConstraintSimplePartiallyApplied}
  \omega\big(ghg^{-1}, b\big) = \omega(h, b)
\end{equation}
for all ${g \in G_b}$ and ${h \in G}$.
Here $G_b$ denotes the stabilizer of $G$ at $b$.

If the group $G$ is abelian,
then each partially applied function
${\omega(-, b) \colon G \rightarrow \R}$ for ${b \in B}$
is entirely unconstrained and we have
${\omega(-, b) = \omega(-, b')}$
for any two ${b, b' \in B}$
in the same orbit of the action ${G \curvearrowright B}$.
In this case,
it comes down to a design choice,
whether we use a filter as $\omega$ that is effectively parametrized
by the orbits of ${G \curvearrowright B}$
(which are points of the quotient space ${G \backslash B}$)
or if we reuse the same weights provided by a single unconstrained filter
${\omega' \colon G \rightarrow \R}$
across all orbits
with the cross-correlation specified by \eqref{eq:xcorrHatGeneralized}.

In particular,
if $G$ is abelian and it acts transitively on $B$,
then a cross-correlation as specified by \eqref{eq:xcorrCheck}
reduces to a cross-correlation as in \eqref{eq:xcorrHatGeneralized}
with a single unconstrained filter
${\omega' \colon G \rightarrow \R}$.
This is in stark contrast
with the previous approach to group convolutions,
where there is a non-trivial constraint on filters
as soon as there are non-trivial stabilizers,
see for example \mbox{\citep[Definition 5]{pmlr-v80-kondor18a}},
\mbox{\citep[Theorem 3.2]{10.5555/3454287.3455107}},
or \mbox{\citep[Section 3.2]{Gerken2023-dm}}.

In general, we may choose some fundamental domain ${D \subset B}$
so the filter $\omega$
is fully described by the family
\begin{equation*}
  \{\omega(-, b) \colon G \rightarrow \R\}_{b \in D}
\end{equation*}
of partially applied functions,
each satisfying the constraint
\eqref{eq:omegaConstraintSimplePartiallyApplied}
for all ${g \in G_b}$ and ${h \in G}$.

\sloppy
Similarly,
the measure $\mu_b$ for ${b \in B}$
is invariant under conjugation by any element
in the stabilizer $G_b$.
So if $G$ acts transitively on $B$,
then it suffices to provide a single locally finite Borel measure $\mu_b$
that is invariant under conjugation by $G_b$
for some choice of basepoint ${b \in B}$.
(In this case we set ${\mu_{g.b} \coloneqq c_{g*} \mu_b}$
for each ${g \in G}$,
which is not an overdetermined definition
by our $G_b$-invariance constraint on $\mu_b$.)

\sloppy
This will be a recurring theme throughout the present paper.
Oftentimes we will have some entity parametrized by $B$
in a $G$-equivariant way entailing a $G_b$-invariance
(or later $G_b$-equivariance)
constraint
for the single entity associated to any ${b \in B}$.
The benefit of working with such parametrized entities is
that it eliminates the need to check
the independence of constructions
from some choice of fundamental domain
(or basepoint for transitive actions),
which can be viewed as an implementation detail.

\subsubsection{Generalization to Vector-Valued Functions}

We also note that the cross-correlation defined in \eqref{eq:xcorrCheck}
readily generalizes to a transformation of vector-valued functions
via matrix-valued filters.
More specifically,
for ${m, n \in \N}$,
for a continuous vector-valued function ${f \colon B \rightarrow \R^n}$,
and for a matrix-valued compactly supported filter
${\omega \colon G \times B \rightarrow \R^{m \times n}}$
satisfying the constraint \eqref{eq:omegaConstraintSimple}
for all ${g, h \in G}$ and ${b \in B}$,
we obtain a cross-correlation
${\omega \checkstar f \colon B \rightarrow \R^m}$
in much the same way.
In particular,
the proof of $G$-equivariance
is provided by the same calculation \eqref{eq:equivarianceFunctionsMod}.
In the remainder of this \cref{sec:xcorr},
we generalize this form of a cross-correlation
to $G$-equivariant vector bundles.

\subsection{Mackey Sections}
\label{sec:mackey}

First of all,
recall from \cref{sec:toNonAbelian}
that $G$ is a Hausdorff topological group
and $B$ a Hausdorff $G$-space.
In order to reduce cross-correlations
transforming sections of vector bundles
to cross-correlations transforming vector-valued functions,
\citet[Section 2.3.1]{10.5555/3454287.3455107}
proposed the use of \emph{Mackey functions}.
In this \cref{sec:mackey} we generalize or adapt this notion
to the case of a not necessarily transitive $G$-action.

To this end,
let ${E \rightarrow B}$ be a $G$-equivariant real vector bundle over $B$.
Roughly speaking,
this amounts to $E$ being a $G$-space,
the bundle projection ${E \rightarrow B}$
being $G$-equivariant,
and each fiber $E_b$ above some ${b \in B}$
(with respect to the projection ${E \rightarrow B}$)
carries the structure of a real vector space.
A continuous map
${f \colon B \rightarrow E}$
with ${f(b) \in E_b}$ for all ${b \in B}$
is henceforth referred to as a continuous \emph{section}
of the bundle ${E \rightarrow B}$,
and we denote the real vector space of such continuous sections
by ${\Gamma(E)}$.

Now let ${f \in \Gamma(E)}$ be a continuous section.
In order to transform $f$
by forming a cross-correlation,
it will be convenient to express the value ${f(h.b) \in E_{h.b}}$
for some ${h \in G}$ and ${b \in B}$
as a vector in $E_b$.
To this end,
we define the map
\begin{equation}
  \label{eq:definitionMackeySection}  
  \tilde{f} \colon
  G \times B \rightarrow E, \,
  (h, b) \mapsto h^{-1}.f(h.b)
\end{equation}
making the diagram
\begin{equation}
  \label{eq:mackeyLift}
  \begin{tikzcd}[ampersand replacement=\&,row sep=6ex, column sep=9ex]
    \&
    E
    \arrow[d]
    \\
    G \times B
    \arrow[r]
    \arrow[ru, "\tilde{f}"]
    \&
    B
    \\[-6ex]
    (h, b)
    \arrow[r, mapsto]
    \&
    b
  \end{tikzcd}
\end{equation}
commute.
Moreover,
$\tilde{f}$ satisfies the two equations
\begin{align}
  \label{eq:periodicity}
  \tilde{f}(h, g.b)
  &=
    g.\tilde{f}(hg, b)
    \quad \text{and}
  \\
  \label{eq:mackeyToSection}
  \tilde{f}(e, b)
  &=
    f(b)
\end{align}
for all ${g, h \in G}$ and ${b \in B}$.

\begin{remark}
  \label{remark:mackeyDetermined}
  The map ${\tilde{f} \colon G \times B \rightarrow E}$
  is completely determined by the equations
  \eqref{eq:periodicity} and \eqref{eq:mackeyToSection}.
  In order to see this,
  let ${h \in G}$ and ${b \in B}$.
  Then we reobtain the defining formula for $\tilde{f}(h, b)$
  provided in
  \eqref{eq:definitionMackeySection}
  only using the equations
  \eqref{eq:periodicity} and \eqref{eq:mackeyToSection}:
  \begin{align*}
    \tilde{f}(h, b)
    &=
      h^{-1}.h.\tilde{f}(eh, b)
    \\
    &\!\!\stackrel{\eqref{eq:periodicity}}{=}
      h^{-1}.\tilde{f}(e, h.b)
    \\
    &\!\!\stackrel{\eqref{eq:mackeyToSection}}{=}
      h^{-1}.f(h.b).
  \end{align*}
\end{remark}

\begin{definition}[Mackey Sections]
  \label{definition:mackeySection}
  We name a lift $\tilde{f}$ as in diagram \eqref{eq:mackeyLift}
  satisfying equation \eqref{eq:periodicity}
  a \emph{Mackey section}
  and we say that $\tilde{f}$ is the
  \emph{Mackey section associated} to the section ${f \in \Gamma(E)}$.
  We denote the real vector space of Mackey sections of $E$
  by $M(E)$.
\end{definition}

If $G$ acts transitively on $B$,
then $\tilde{f}$ (and hence $f$)
is uniquely determined by the partially applied function
${\tilde{f}(-, b) \colon G \rightarrow E_b \cong \R^n}$
for any choice of basepoint ${b \in B}$.
Moreover,
the equation \eqref{eq:periodicity}
provides the $G_b$-equivariance constraint
\begin{equation*}
  \tilde{f}(h, b)
  =
  g.\tilde{f}(hg, b)
\end{equation*}
for any ${g \in G_b}$ and ${h \in G}$.
This is the type of function ${G \rightarrow E_b \cong \R^n}$
that \citet[Section 2.3.1]{10.5555/3454287.3455107}
refer to as a \emph{Mackey function},
i.e. a function ${G \rightarrow E_b \cong \R^n}$ is a Mackey function
(and hence determines a section of $E$)
iff it satisfies such $G_b$-equivariance constraint.
So the Mackey section ${\tilde{f} \in M(E)}$ can also be thought of
as the family
${\big\{\tilde{f}(-, b) \colon G \rightarrow E_b\big\}_{b \in B}}$
of Mackey functions
with respect to any basepoint in $B$
associated to the section ${f \in \Gamma(E)}$.

\subsubsection{Group Actions on (Mackey) Sections}

We now define compatible $G$-actions on sections and Mackey sections
of the vector bundle ${E \rightarrow B}$.
The associated $G$-action on ordinary sections
can be seen as a form of conjugation:
\begin{equation*}
  \_.\_ \colon
  G \times \Gamma(E) \rightarrow \Gamma(E), \,
  (g, f) \mapsto g.f,
\end{equation*}
where
\begin{equation}
  \label{eq:conjAction}
  (g.f)(b) \coloneqq g.f\big(g^{-1}.b\big)
\end{equation}
for all ${b \in B}$.
As already noted above,
in the present paper
function application takes precedence over the binary
operator
\enquote{$\_.\_$} denoting group actions.
The corresponding $G$-action on Mackey sections is more simple:
\begin{equation*}
  \_.\_ \colon
  G \times M(E) \rightarrow M(E), \,
  \big(g, \tilde{f}\big) \mapsto g.\tilde{f},
\end{equation*}
where
\begin{equation}
  \label{eq:mackeyAction}
  \big(g.\tilde{f}\big)(h, b) \coloneqq \tilde{f}\big(g^{-1}h, b\big)
\end{equation}
for all ${h \in G}$ and ${b \in B}$.

\begin{proposition}
  \label{thm:isoMackey}
  The map
  \begin{equation}
    \label{eq:sectionToMackeySection}
    \Gamma(E) \rightarrow M(E), \,
    f \mapsto \tilde{f}
  \end{equation}
  sending a section ${f \in \Gamma(E)}$
  to its associated Mackey section ${\tilde{f} \in M(E)}$
  is an $\R$-linear $G$-equivariant bijection with inverse
  \begin{equation}
    \label{eq:mackeySectionToSection}
    M(E) \rightarrow \Gamma(E), \,
    \tilde{f} \mapsto \tilde{f}(e, -).
  \end{equation}
\end{proposition}

\begin{proof}
  Let ${f \in \Gamma(E)}$
  and let ${\tilde{f} \in M(E)}$
  be the associated Mackey section.
  By the equation \eqref{eq:mackeyToSection},
  the partial evaluation at the neutral element
  \eqref{eq:mackeySectionToSection}
  is a left inverse to the map \eqref{eq:sectionToMackeySection}.
  As \eqref{eq:periodicity} is the defining equation
  of the \cref{definition:mackeySection}
  of Mackey sections
  and as $\tilde{f}$ is completely determined by equations
  \eqref{eq:periodicity} and \eqref{eq:mackeyToSection}
  considering \cref{remark:mackeyDetermined},
  the map \eqref{eq:mackeySectionToSection} is a right inverse as well.
  For the proof of $G$-equivariance
  let ${g, h \in G}$, and ${b \in B}$.
  Then we have
  \begin{align*}
    \big(\widetilde{g.f}\big)(h, b)
    &\stackrel{\eqref{eq:definitionMackeySection}}{=}
      h^{-1}.(g.f)(h.b)
    \\
    &\stackrel{\eqref{eq:conjAction}}{=}
      h^{-1}.g.f\big(g^{-1}.h.b\big)
    \\
    &~=
      \big(g^{-1}h\big)^{-1}.f\big(g^{-1}h.b\big)
    \\
    &\stackrel{\eqref{eq:definitionMackeySection}}{=}
      \tilde{f}\big(g^{-1}h, b\big)
    \\
    &\stackrel{\eqref{eq:mackeyAction}}{=}
      g.\tilde{f}(h, b).
      \qedhere
  \end{align*}
\end{proof}

\subsection{Cross-Correlations with a Filter}
\label{sec:xcorrFilter}

First of all,
recall from \cref{sec:toNonAbelian} that
\[\{\mu_b \colon \mathcal{B}(G) \rightarrow [0, \infty]\}_{b \in B}\]
is a continuous family of locally finite Borel measures
such that
\begin{equation}
  \mu_{g.b} = c_{g*} \mu_b
  \tag{\ref{eq:muInv} revisited}
\end{equation}
for any ${g \in G}$ and ${b \in B}$,
where
${c_{g*} \mu_b}$ is the pushforward measure
of $\mu_b$ along the conjugation
\begin{equation*}
  c_g \colon G \rightarrow G, \,
  h \mapsto g h g^{-1}.
\end{equation*}
Now
suppose we have $G$-equivariant real vector bundles
${E \rightarrow B}$ and
${F \rightarrow B}$.
We aim to transform a section ${f \in \Gamma(E)}$
to a section of ${F \rightarrow B}$.
So for each point ${b \in B}$
we need to give a vector in $F_b$
in terms of $f$.
When doing this by a cross-correlation generalizing \eqref{eq:xcorrCheck},
we obtain such a vector in $F_b$
as a weighted sum or integral
over the vector-values of the \enquote{Mackey function}
\begin{equation*}
  \tilde{f}(-, b) \colon G \rightarrow E_b, \,
  h \mapsto h^{-1}.f(h.b)
  .
\end{equation*}
More specifically,
the \emph{filter} $\omega$
gives a linear map
\begin{equation}
  \label{eq:omegaAsMap}
  \omega(h, b) \colon E_b \rightarrow F_b
\end{equation}
for each ${b \in B}$ and ${h \in G}$ so a value in $F_b$
can be obtained as an integral
\begin{equation}
  \label{eq:omegaIntegralInFiber}
  \int_G \omega(h, b)\big(\tilde{f}(h, b)\big) d\mu_b(h) \in F_b.
\end{equation}

In order to formalize the idea that $\omega$ is continuous
as an assignment of linear maps \eqref{eq:omegaAsMap},
we use the homomorphism bundle
${\mathrm{Hom}(E, F) \rightarrow B}$
whose fiber above ${b \in B}$ is the vector space of linear maps
${E_b \rightarrow F_b}$.
So formally,
we assume that $\omega$ is a continuous lift
in the commutative diagram
\begin{equation}
  \label{eq:omegaLift}
  \begin{tikzcd}[ampersand replacement=\&,row sep=6ex]
    \&
    \mathrm{Hom}(E, F)
    \arrow[d]
    \\
    G \times B
    \arrow[r]
    \arrow[ru, "\omega"]
    \&
    B
    \\[-6ex]
    (h, b)
    \arrow[r, mapsto]
    \&
    b.
  \end{tikzcd}
\end{equation}
Now in order for the integral \eqref{eq:omegaIntegralInFiber}
to be well-defined
we impose that ${\omega(-, b) \colon G \rightarrow \mathrm{Hom}(E_b, F_b)}$
has compact support for any ${b \in B}$
and in order for the map
${\Gamma(E) \rightarrow \Gamma(F)}$
defined by $\omega$ to be $G$-equivariant,
we impose the constraint
\begin{equation}
  \label{eq:omegaConstraint}
  \omega\big(ghg^{-1}, g.b\big)(g.v) = g.\omega(h, b)(v)
\end{equation}
for all ${g, h \in G}$, ${b \in B}$, and ${v \in E_b}$.
Thus,
if ${D \subseteq B}$ is some fundamental domain,
then the filter $\omega$ is fully determined
by all partially applied maps
\begin{equation*}
  \omega(-, b) \colon G \rightarrow
  \mathrm{Hom}(E_b, F_b) \cong \R^{m \times n}
\end{equation*}
for ${b \in D}$ (where ${m, n \in \N}$).
Moreover,
for any one such partially applied map
the equation \eqref{eq:omegaConstraint} yields the constraint
\begin{equation*}
  \omega\big(ghg^{-1}, b\big)(g.v) = g.\omega(h, b)(v)
\end{equation*}
for all ${g \in G_b}$, ${h \in G}$, and ${v \in E_b}$.

Now in light of \cref{thm:isoMackey},
to describe a linear $G$-equivariant map
${\Gamma(E) \rightarrow \Gamma(F)}$
using $\omega$ 
we may as well describe a linear $G$-equivariant map
${M(E) \rightarrow M(F)}$.
This appears to be a sensible choice,
considering the use of $\tilde{f}$
in the integral \eqref{eq:omegaIntegralInFiber}.

\begin{definition}[Group Cross-Correlations]
  \label{definition:xcorr}
  For a Mackey section ${\tilde{f} \in M(E)}$
  we define the
  \emph{cross-correlation}
  ${\omega \star \tilde{f} \in M(F)}$
  by
  \begin{equation}
    \label{eq:xcorr}
    \big(\omega \star \tilde{f}\big)(h, b) \coloneqq
    \int_G \omega(k, b)\big(\tilde{f}(hk, b)\big) d\mu_b(k)
  \end{equation}
  for all ${h \in G}$ and ${b \in B}$.
\end{definition}

\begin{theorem}[Well-Defined Cross-Correlations]
  \label{thm:isMackeySection}
  For ${\tilde{f} \in M(E)}$
  the cross-correlation
  ${\omega \star \tilde{f}}$
  is indeed a Mackey section
  in the sense of \cref{definition:mackeySection}.
\end{theorem}

\begin{proof}
  By construction of ${\omega \star \tilde{f}}$
  we have the commutative diagram
  \begin{equation*}
    \begin{tikzcd}[ampersand replacement=\&,row sep=6ex, column sep=9ex]
      \&
      F
      \arrow[d]
      \\
      G \times B
      \arrow[r]
      \arrow[ru, "\omega \star \tilde{f}"]
      \&
      B
      \\[-6ex]
      (h, b)
      \arrow[r, mapsto]
      \&
      b.
    \end{tikzcd}
  \end{equation*}
  Now let ${g, h \in G}$ and ${b \in B}$.
  In order to complete our proof,
  we have to show the equation \eqref{eq:periodicity}
  with ${\omega \star \tilde{f}}$ substituted for
  ${\tilde{f}}$.
  And indeed we have
  \begin{align*}
    \big(\omega \star \tilde{f}\big)(h, g.b)
    &\stackrel{\eqref{eq:xcorr}}{=}
      \int_G \omega(k, g.b)\big(\tilde{f}(hk, g.b)\big) d\mu_{g.b}(k)
    \\
    &\stackrel{\eqref{eq:periodicity}}{=}
      \int_G \omega(k, g.b)\big(g.\tilde{f}(hkg, b)\big) d\mu_{g.b}(k)
    \\
    &\,\stackrel{\eqref{eq:muInv}}{=}
      \int_G \omega(k, g.b)\big(g.\tilde{f}(hkg, b)\big) dc_{g*} \mu_{b}(k)
    \\
    &~=
      \int_G \omega\big(gkg^{-1}, g.b\big)\big(g.\tilde{f}(hgk, g.b)\big) d\mu_{b}(k)
    \\
    &\stackrel{\eqref{eq:omegaConstraint}}{=}
      \int_G g.\omega(k, b)\big(\tilde{f}(hgk, g.b)\big) d\mu_{b}(k)
    \\
    &~=
      g.\int_G \omega(k, b)\big(\tilde{f}(hgk, g.b)\big) d\mu_{b}(k)
    \\
    &\stackrel{\eqref{eq:xcorr}}{=}
      g.\big(\omega \star \tilde{f}\big)(hg, b).
  \end{align*}
  Here the fourth equality follows from the change of variables formula
  for the pushforward measure $c_{g*}\mu_b$.
\end{proof}

\begin{remark}
  If the measure ${\mu_b \colon \mathcal{B}(G) \rightarrow [0, \infty]}$
  is left-invariant for all ${b \in B}$,
  then we may write the cross-correlation \eqref{eq:xcorr} also as
  \begin{equation}
    \label{eq:xcorrLeftInv}
    \begin{split}
      \big(\omega \star \tilde{f}\big)(h, b)
      &=
        \int_G \omega(k, b)\big(\tilde{f}(hk, b)\big) d\mu_b(k)
      \\
      &=
        \int_G \omega\big(h^{-1} k, b\big)\big(\tilde{f}(k, b)\big) d\mu_b(k)
    \end{split}
  \end{equation}
  for all ${h \in G}$ and ${b \in B}$,
  which is an adaptation of the formula provided by
  \mbox{\citep[Equation 7]{10.5555/3454287.3455107}} and
  \mbox{\citep[Definition 3.8]{Gerken2023-dm}}
  to the case of a not necessarily transitive group action.
  By defining
  \begin{equation*}
    \omega' \colon G \times B \rightarrow \mathrm{Hom}(E, F), \,
    (h, b) \mapsto \omega\big(h^{-1}, b\big)
  \end{equation*}
  we may then also write the cross-correlation \eqref{eq:xcorr}
  as a \emph{convolution}:
  \begin{equation*}
    \big(\omega' \ast \tilde{f}\big)(h, b) \coloneqq
    \int_G \omega'\big(k^{-1}h, b\big)\big(\tilde{f}(k, b)\big) d\mu_b(k)
  \end{equation*}
  for ${h \in G}$ and ${b \in B}$.
  Note the subtle difference in notation with
  $\ast$ substituted for $\star$.
  Indeed,
  we have
  \begin{align*}
    \big(\omega' \ast \tilde{f}\big)(h, b)
    &=
      \int_G \omega'\big(k^{-1}h, b\big)\big(\tilde{f}(k, b)\big) d\mu_b(k)
    \\
    &=
      \int_G \omega\big(h^{-1}k, b\big)\big(\tilde{f}(k, b)\big) d\mu_b(k)
    \\
    &\!\!\stackrel{\eqref{eq:xcorrLeftInv}}{=}
      \big(\omega \star \tilde{f}\big)(h, b).
  \end{align*}
  However,
  as this can only be done for left-invariant measures,
  we stick with \cref{definition:xcorr}
  albeit the clumsy wording.
\end{remark}

As we defined this general notion of a cross-correlation
as a transformation outputting Mackey sections,
the proof of $G$-equivariance is a straightforward calculation.

\begin{lemma}[Equivariance of Cross-Correlations]
  \label{thm:xcorrEquivariance}
  The map
  \[\omega \star - \colon M(E) \rightarrow M(F), \,
    \tilde{f} \mapsto \omega \star \tilde{f}\]
  is $G$-equivariant.
\end{lemma}

\begin{proof}
  For ${\tilde{f} \in M(E)}$,
  ${g, h \in G}$, and ${b \in B}$ we have
  \begin{align*}
    \big(\omega \star g.\tilde{f}\big)(h, b)
    &=
      \int_G \omega(k, b)\big((g.\tilde{f})(hk, b)\big) d\mu_b(k)
    \\
    &=
      \int_G \omega(k, b)\big(\tilde{f}\big(g^{-1}hk, b\big)\big) d\mu_b(k)
    \\
    &=
    \big(\omega \star \tilde{f}\big)\big(g^{-1}h, b\big)
    \\
    &=
      \big(g.\big(\omega \star \tilde{f}\big)\big)(h, b).
      \qedhere
  \end{align*}
\end{proof}

\section{Orbitwise Integral Transforms}
\label{sec:orbitwiseIntegralTrafos}

In the previous \cref{sec:xcorrFilter}
we introduced a form of cross-correlation
transforming sections of some $G$-equivariant real vector bundle
${E \rightarrow B}$
to sections of another such vector bundle
${F \rightarrow B}$.
In the remainder of this paper,
we compare this notion of a cross-correlation
to that of
an integral transform of sections
${T_{\kappa} \colon \Gamma(E) \rightarrow \Gamma(F)}$
for some kernel $\kappa$.
Informally,
a kernel $\kappa$ is an assignment of a linear map
\begin{equation}
  \label{eq:kappaAsMap}
  \kappa(c, b) \colon E_c \rightarrow F_b
\end{equation}
to any ${b \in B}$ and any $c$ within the receptive field of $b$;
so the value of the integral transform ${T_{\kappa}(f)}$ of a section
${f \in \Gamma(E)}$ at ${b \in B}$ can be written as
\begin{equation}
  \label{eq:kappaIntegralInFiber}
  T_{\kappa}(f)(b) = \int \kappa(c, b)(f(c)) dc \in F_b
\end{equation}
with the domain of integration and its measure to be determined.

Now suppose ${f \in \Gamma(E)}$ is a section of $E$
and that ${\omega \colon G \times B \rightarrow \mathrm{Hom}(E, F)}$
is a filter as in the previous \cref{sec:xcorrFilter}.
Then the value of the resulting section
${\big(\omega \star \tilde{f}\big)(e, -) \in \Gamma(F)}$
at some point ${b \in B}$ can be written as
\begin{equation*}
  \int_G \omega(h, b)\big(\tilde{f}(h, b)\big) d\mu_b(h) \in F_b.
  \tag{\ref{eq:omegaIntegralInFiber} revisited}
\end{equation*}
Moreover,
as the Mackey function
${\tilde{f}(-, b) \colon G \rightarrow E_b}$
only sees values of $f$ at points that are in the same orbit as $b$,
the receptive field of $b$ is constrained to its orbit ${G.b \subseteq B}$.
So in order to obtain an integral transform $T_{\kappa}$
comparable to cross-correlations in the sense of \cref{definition:xcorr},
we assume $G.b$ to be the domain of integration
in \eqref{eq:kappaIntegralInFiber}
and the kernel $\kappa$
to be defined on
\begin{equation*}
  \{(c, b) \in B \times B \mid c \in G.b\} =
  \bigsqcup_{b \in B} G.b,
\end{equation*}
where the right-hand side denotes the
disjoint union
as a set endowed with the subspace topology of ${B \times B}$.
Then in order to use $G.b$ as a domain of integration
in \eqref{eq:kappaIntegralInFiber},
we also need
a locally finite Borel measure
${\bar{\mu}_b \colon \mathcal{B}(G.b) \rightarrow [0, \infty]}$.
So we also assume that we have a family
\begin{equation*}
  \{\bar{\mu}_b \colon \mathcal{B}(G.b) \rightarrow [0, \infty]\}_{b \in B}
\end{equation*}
of locally finite Borel measures.
As the scope of this paper is confined to
$G$-equivariant integral transforms,
we further impose the equation
\begin{equation}
  \label{eq:muBarInv}
  \bar{\mu}_{g.b} = (g.\_)_* \bar{\mu}_b
\end{equation}
for all ${g \in G}$ and ${b \in B}$,
where ${(g.\_)_* \bar{\mu}_b}$ is the pushforward measure
of $\bar{\mu}_b$ along the self-homeomorphism
\[g.\_ \colon B \rightarrow B, \,
  b \mapsto g.b.\]
In particular,
the measure $\bar{\mu}_b$ is $G_b$-invariant
for any ${b \in B}$.
For an in depth discussion
on how such families of measures
satisfying even more restrictive constraints
such as $G$-invariance
can be obtained in a natural way,
consider the Appendix \ref{sec:famInvMeas}.

Now in order to formalize the idea that $\kappa$ is continuous
as an assignment of linear maps \eqref{eq:kappaAsMap},
we view the natural surjections
${E \times B \rightarrow B \times B}$ and
${B \times F \rightarrow B \times B}$
as vector bundles over ${B \times B}$
and we assume that $\kappa$ is a continuous lift
in the commutative diagram
\begin{equation*}
  \begin{tikzcd}[ampersand replacement=\&,row sep=6ex]
    \&
    \mathrm{Hom}(E \times B, B \times F)
    \arrow[d]
    \\
    \displaystyle{\bigsqcup_{b \in B} G.b}
    \arrow[r, hook]
    \arrow[ru, "\kappa"]
    \&
    B \times B\!
  \end{tikzcd}
\end{equation*}
with each partially applied lift
${\kappa(-, b) \colon G.b \rightarrow \mathrm{Hom}(E, F_b)}$
for ${b \in B}$ compactly supported.
The associated \emph{orbitwise integral transform}
${T_{\kappa} \colon \Gamma(E) \rightarrow \Gamma'(F), \,
  f \mapsto T_{\kappa}(f)}$
is then defined by
\begin{equation}
  \label{eq:orbitwiseIntegralTrafo}
  T_{\kappa}(f) \colon B \rightarrow F, \,
  b \mapsto \int_{G.b} \kappa(c, b)(f(c)) d\bar{\mu}_b(c),
\end{equation}
where $\Gamma'(F)$ denotes the vector space
of all not necessarily continuous sections of the vector bundle
${F \rightarrow B}$.
As we did not impose any relation or compatibility
on the measures $\bar{\mu}_b$ and $\bar{\mu}_c$
for ${G.b \neq G.c}$,
we cannot guarantee continuity for sections obtained
as the output of $T_{\kappa}$.
However,
if we can write
${T_{\kappa} \colon \Gamma(E) \rightarrow \Gamma'(F)}$
using cross-correlations,
as we will discuss in \cref{sec:integralTrafosAsXCorr},
then the continuity of the sections it outputs follows a posteriori.

\subsection{Equivariance of Orbitwise Integral Transforms}
\label{sec:equivIntegralTrafo}

In the context of $G$-invariant measures on the base space $B$,
constraints on kernels entailing their integral transforms
be $G$-equivariant have been widely studied.
In the following we show that essentially the same constraint
as provided by \mbox{\citep[Section 4.2]{Gerken2023-dm}}
is sufficient and under suitable tameness assumptions also necessary
for an orbitwise integral transform
as defined by \eqref{eq:orbitwiseIntegralTrafo}
to be $G$-equivariant.
More specifically, this constraint on a kernel $\kappa$ as above is
that we have the equation
\begin{equation}
  \label{eq:kappaConstraint}
  g.\kappa(c, b)(v) = \kappa(g.c, g.b)(g.v)
\end{equation}
for all ${g \in G}$, ${b \in B}$, ${c \in G.b}$, and ${v \in E_c}$.

\begin{lemma}
  \label{thm:kappaConstraintImpliesEquivariance}
  If we have equation \eqref{eq:kappaConstraint}
  for all ${g \in G}$, ${b \in B}$, ${c \in G.b}$, and ${v \in E_c}$,
  then the integral transform
  ${T_{\kappa} \colon \Gamma(E) \rightarrow \Gamma'(F)}$
  is $G$-equivariant.
\end{lemma}

\begin{proof}
  For ${f \in \Gamma(E)}$, ${g \in G}$ and ${b \in B}$
  let ${b' \coloneqq g^{-1}.b}$.
  Then we have
  \begin{align*}
    T_{\kappa}(g.f)(b)
    &=
      \int_{G.b} \kappa(c, b)((g.f)(c)) d\bar{\mu}_b(c)
    \\[1ex]
    &=
      \int_{G.b} \kappa(c, b)\big(g.f\big(g^{-1}.c\big)\big) d\bar{\mu}_b(c)
    \\[1ex]
    &=
      \int_{G.b} \kappa\big(gg^{-1}.c, gg^{-1}.b\big)\big(
      g.f\big(g^{-1}.c\big)
      \big) d\bar{\mu}_b(c)
    \\[1ex]
    &\!\!\stackrel{\eqref{eq:kappaConstraint}}{=}
      \int_{G.b} g.\kappa\big(g^{-1}.c, g^{-1}.b\big)\big(
      f\big(g^{-1}.c\big)
      \big) d\bar{\mu}_b(c)
    \\[1ex]
    &=
      g.\int_{G.b} \kappa\big(g^{-1}.c, b'\big)\big(
      f\big(g^{-1}.c\big)
      \big) d\bar{\mu}_{g.b'}(c)
    \\[1ex]
    &\!\!\stackrel{\eqref{eq:muBarInv}}{=}
      g.\int_{G.b} \kappa\big(g^{-1}.c, b'\big)\big(
      f\big(g^{-1}.c\big)
      \big) d(g.\_)_*\bar{\mu}_{b'}(c)
    \\[1ex]
    &=
      g.\int_{G.b} \kappa(c, b')(f(c)) d\bar{\mu}_{b'}(c)
    \\[1ex]
    &=
      g.T_{\kappa}(f)(b')
    \\[1ex]
    &=
      g.T_{\kappa}(f)\big(g^{-1}.b\big)
    \\[1ex]
    &=
      \big(g.T_{\kappa}(f)\big)(b).
      \qedhere
  \end{align*}
\end{proof}

\begin{proposition}
  \label{thm:equivToConstraint}
  Suppose that the integral transform
  ${T_{\kappa} \colon \Gamma(E) \rightarrow \Gamma'(F)}$
  is $G$-equivariant.
  Moreover,
  for any ${b \in B}$
  we assume that the measure
  ${\bar{\mu}_b \colon \mathcal{B}(G.b) \rightarrow [0, \infty]}$
  is strictly positive,
  that $G.b$ is paracompact,
  and that any compactly supported continuous section
  of the restricted vector bundle
  ${E |_{G.b} \rightarrow G.b}$
  has a continuous extension to a section of ${E \rightarrow B}$,
  where ${E |_{G.b} \coloneqq \bigcup_{c \in G.b} E_c}$.
  Then we have the equation \eqref{eq:kappaConstraint}
  for all ${g \in G}$, ${b \in B}$, ${c \in G.b}$, and ${v \in E_c}$.
\end{proposition}

\begin{proof}
  Let ${f \in \Gamma(E)}$, ${b \in B}$, and ${g \in G}$.
  Then we have
  \begin{equation}
    \label{eq:equivToConstraintInt}
    \begin{split}
      \int_{G.b}
      g.\kappa(c, b)(f(c)) d\bar{\mu}_b(c)
      &=
        g.\int_{G.b}
        \kappa(c, b)(f(c)) d\bar{\mu}_b(c)
      \\
      &=
        g.T_{\kappa}(f)(b)
      \\
      &=
        \big(g.T_{\kappa}(f)\big)(g.b)
      \\
      &=
        T_{\kappa}(g.f)(g.b)
      \\
      &=
        \int_{G.b}
        \kappa(c, g.b)((g.f)(c)) d\bar{\mu}_{g.b}(c)
      \\
      &\!\!\stackrel{\eqref{eq:muBarInv}}{=}
        \int_{G.b}
        \kappa(c, g.b)\big(
        g.f\big(g^{-1}.c\big)
        \big) d(g.\_)_* \bar{\mu}_{b}(c)
      \\
      &=
        \int_{G.b}
        \kappa(g.c, g.b)(g.f(c)) d\bar{\mu}_{b}(c).
    \end{split}
  \end{equation}
  Now for ${c \in G.b}$ let
  \begin{equation*}
    \xi(c) \colon E_c \rightarrow F_{g.b}, \,
    v \mapsto
    \xi(c)(v) \coloneqq g.\kappa(c, b)(v) - \kappa(g.c, g.b)(g.v).
  \end{equation*}
  By the previous equation \eqref{eq:equivToConstraintInt}
  we have
  \begin{equation}
    \label{eq:xiVanishing}
    \int_{G.b}
    \xi(c)(f(c)) d\bar{\mu}_{b}(c)
    = 0
  \end{equation}
  for all sections ${f \in \Gamma(E)}$
  and we have to show that
  ${\xi(c)(v) = 0}$ for all ${c \in B}$ and ${v \in E_c}$.
  To this end,
  it suffices to show that
  \begin{equation*}
    (\alpha \circ \xi(c))(v) = 0
  \end{equation*}
  for all linear forms ${\alpha \colon F_{g.b} \rightarrow \R}$,
  ${c \in B}$, and ${v \in E_c}$.
  Now let
  ${\alpha \colon F_{g.b} \rightarrow \R}$
  be a linear form
  and let ${\sigma \in \Gamma_c(E^* |_{G.b})}$
  be defined by
  \[\sigma(v) \coloneqq (\alpha \circ \xi(c))(v)\]
  for all ${c \in G.b}$ and ${v \in E_c}$,
  where
  ${E^* |_{G.b} \coloneqq \bigcup_{c \in G.b} E_c^*}$
  is the restricted bundle of dual spaces.
  Moreover, suppose we have a
  compactly supported continuous section ${f \in \Gamma_c(E |_{G.b})}$
  of the restricted vector bundle
  ${E |_{G.b} \rightarrow G.b}$.
  By assumption there is a continuous extension
  ${\hat{f} \in \Gamma(E)}$ of $f$ to a section defined on all of $B$.
  Furthermore, we have
  \begin{align*}
    0
    &=
      \alpha(0)
    \\
    &\!\!\stackrel{\eqref{eq:xiVanishing}}{=}
      \alpha\left(
      \int_{G.b}
      \xi(c)\big(\hat{f}(c)\big) d\bar{\mu}_{b}(c)
      \right)
    \\
    &=
      \int_{G.b}
      (\alpha \circ \xi(c))\big(\hat{f}(c)\big) d\bar{\mu}_{b}(c)
    \\
    &=
      \int_{G.b}
      \sigma(f(c)) d\bar{\mu}_{b}(c)
      .
  \end{align*}
  Thus, we conclude from \cref{thm:vanishingDualSections}
  that ${\sigma = 0}$
  and hence
  \begin{equation*}
    (\alpha \circ \xi(c))(v) = \sigma(v) = 0
  \end{equation*}
  for all ${c \in B}$ and ${v \in E_c}$.
\end{proof}

\begin{corollary}
  \label{thm:equivToConstraintTransitive}
  Suppose the action ${G \curvearrowright B}$
  is transitive,
  that $B$ is paracompact,
  and that ${\bar{\mu}_b \colon \mathcal{B}(B) \rightarrow [0, \infty]}$
  is strictly positive for some (hence any) ${b \in B}$.
  If
  ${T_{\kappa} \colon \Gamma(E) \rightarrow \Gamma'(F)}$
  is $G$-equivariant,
  then we have the equation \eqref{eq:kappaConstraint}
  for all ${g \in G}$, ${b \in B}$, ${c \in G.b}$, and ${v \in E_c}$.  
\end{corollary}

\section{Integral Transforms as Cross-Correlations}
\label{sec:integralTrafosAsXCorr}

In this \cref{sec:integralTrafosAsXCorr}
we establish a close relationship
between orbitwise integral transforms
and cross-correlations.
In particular,
we will provide a construction
for writing a $G$-equivariant orbitwise integral transform
associated to a kernel $\kappa$
as a cross-correlation with a filter $\omega$.
As it turns out,
such a filter $\omega$ may not be fully determined by $\kappa$.
So in general,
lifting an integral transform to a cross-correlation
requires some choices to be made.
Before we discuss this in full generality,
we demonstrate at a simple example
how this may require some trade-offs.

\subsection{Example with Real Numbers and Integers}
\label{sec:example}

Let ${B \coloneqq \R}$ and let ${G \coloneqq \R \times \Z}$.
We define a $G$-action on $B$ by setting
\begin{equation*}
  g.b \coloneqq g_1 + g_2 + b
\end{equation*}
for all ${g = (g_1, g_2) \in \R \times \Z}$ and ${b \in \R}$.
Moreover, we assume we have identical trivial vector bundles
\begin{equation*}
  E \coloneqq B \times \R \eqqcolon F \rightarrow B, \,
  (b, v) \mapsto b;
\end{equation*}
so we also identify their sections with continuous functions
${B \rightarrow \R}$.
Furthermore,
the action on ${E = F}$
is confined to the first component:
\begin{equation*}
  g.(b, v) \coloneqq (g.b, v)
\end{equation*}
for ${g \in G}$ and ${(b, v) \in \R \times \R = E}$.
So for a function
${f \colon B \rightarrow \R}$
as a section of $E$
its associated Mackey section can be identified with the function
\begin{equation*}
  \tilde{f} \colon G \times B \rightarrow \R, \,
  (g, b) \mapsto f(g.b).
\end{equation*}
For ${b, c \in B}$
the kernel $\kappa$ provides a linear map
${\kappa(c, b) \colon \R \rightarrow \R}$,
which we identify with the corresponding scalar coefficient
making the kernel a continuous function
${\kappa \colon B \times B \rightarrow \R}$.
Similarly,
we view any filter $\omega$ as a function
${\omega \colon G \times B \rightarrow \R}$.

Now let
${\bar{\mu} \colon \mathcal{B}(\R) \rightarrow [0, \infty]}$
to be the Lebesgue measure
and let
${\mu \colon \mathcal{B}(\R \times \Z) \rightarrow [0, \infty]}$
be the product measure of the Lebesgue measure on $\R$
and the counting measure on $\Z$.
As the Lebesgue measure also is a Haar measure
and as such translation invariant,
the constant family of measures
${\{\bar{\mu}_b \coloneqq \bar{\mu}\}_{b \in B}}$
satisfies the required constraint \eqref{eq:muBarInv}.
Moreover,
as $G$ is abelian
the constant family of measures
${\{\mu_b \coloneqq \mu\}_{b \in B}}$
satisfies the constraint \eqref{eq:muInv}.

Now if we had a filter
${\omega \colon G \times B \rightarrow \R}$
with the desired properties,
then in particular the equation
\begin{equation}
  \label{eq:exampleComparison}
  \begin{split}
    \int_G \omega(h, b)f(h.b) d\mu(h)
    &=
      \int_G \omega(h, b)\tilde{f}(h, b) d\mu(h)
    \\
    &=
      \big(\omega \star \tilde{f}\big)(e, b)
    \\
    &=
      T_{\kappa}(f)(b)
    \\
    &=
      \int_B \kappa(c, b) f(c) d\bar{\mu}(c)
  \end{split}  
\end{equation}
would be satisfied for all continuous functions
${f \colon B \rightarrow \R}$
and all ${b \in B}$.
So in order to find a filter $\omega$
satisfying \eqref{eq:exampleComparison},
we may continuously choose for each pair
${(c, b) \in \supp \kappa}$
an element ${\theta(c, b) \in G}$
with
\begin{equation}
  \label{eq:thetaLift}
  c = \theta(c, b).b
\end{equation}
and set
\begin{equation}
  \label{eq:omega_theta}
  \omega(\theta(c, b), b) \coloneqq \kappa(c, b)
\end{equation}
with all other values of $\omega$ set to $0$.
To this end,
we could define the continuous map
\begin{equation}
  \label{eq:thetaGlobal}
  \theta \colon B \times B \rightarrow G,\,
  (c, b) \mapsto 
  (c-b, 0),
\end{equation}
which works for any kernel
${\kappa \colon B \times B \rightarrow \R}$
irrespective of its support.

\subsubsection{Special Support of Kernel}

Now let
\[S_i \coloneqq
  \{(c, b) \in \R \times \R \mid |c-b-i| \leq \varepsilon\}
\]
for ${i = -1,0,1}$ and some small ${0 < \varepsilon < \frac{1}{2}}$
and suppose we have
\begin{equation*}
  \supp \kappa \subseteq R \coloneqq S_{-1} \cup S_0 \cup S_1
\end{equation*}
as pictured in \cref{fig:supp_kappa}.
\begin{figure}[t]
  \centering
  \supportGraphic
  \caption{The region $R$ containing the support of $\kappa$ shaded in red.}
  \label{fig:supp_kappa}
\end{figure}
In this case
we may choose for $\theta$ the continuous map 
\begin{equation}
  \label{eq:thetaSpecialCase}
  \theta \colon R \rightarrow G,\,
  (c, b) \mapsto
  \begin{cases}
    (c-b-1,  1) & (c, b) \in S_1 \\
    (c-b  ,  0) & (c, b) \in S_0 \\
    (c-b+1, -1) & (c, b) \in S_{-1}.
  \end{cases}
\end{equation}

As $G$ is abelian and as the action ${G \curvearrowright B}$ is transitive,
we have ${\omega(-, b) = \omega(-, b')}$
by the constraint \eqref{eq:omegaConstraint}
for all ${b, b' \in B}$,
see also \cref{sec:contraintsInvariances}.
Now let
${\omega' \coloneqq \omega(-, b) \colon G \rightarrow \R}$
%
for some (hence any) ${b \in B}$.
We consider the support of the filter $\omega'$
depending on our choice for $\theta$.
If we use $\theta$ as specified in \eqref{eq:thetaGlobal}
to define $\omega$ using equation \eqref{eq:omega_theta},
then we have
\begin{equation}
  \label{eq:omegaSuppGlobal}
  \supp \omega' \subseteq
  \bigcup_{i=-1,0,1} [i-\varepsilon, i+\varepsilon] \times \{0\}
\end{equation}
whereas if we use $\theta$ as specified in \eqref{eq:thetaSpecialCase},
then we have
\begin{equation}
  \label{eq:omegaSuppSpecialCase}
  \supp \omega' \subseteq
  [-\varepsilon, \varepsilon] \times \{-1, 0, 1\}.
\end{equation}
In case we have \eqref{eq:omegaSuppSpecialCase} for the support,
then any discretization of $\omega'$
can be represented by a fully populated 2D array,
which is untrue for \eqref{eq:omegaSuppGlobal}.

So clearly,
there is a trade-off to be made here.
On the one hand,
we have the construction using \eqref{eq:thetaGlobal}
that works irrespective of the support of $\kappa$,
and on the other hand
there is the construction using \eqref{eq:thetaSpecialCase}
that only works for
${\supp \kappa \subseteq R}$
but it has the benefit that $\omega'$ can be discretized
by a fully populated 2D array.
For this reason,
our general construction of the filter $\omega$
will not only depend on the kernel $\kappa$ itself
but also on a choice as we had it here with $\theta$.

\subsubsection{Comparison to \enquote{Bi-Equivariant Kernels}}
\label{sec:comparison}

Before we proceed with the construction of filters from kernels,
we use the above example,
to compare the present notion of a group cross-correlation
to the approach by
\citet{10.5555/3454287.3455107},
which is also surveyed in \mbox{\citep[Section 3.2]{Gerken2023-dm}}.
In their terminology
the function
${\omega' \colon G \rightarrow \R}$
is a \enquote{one-argument kernel}
and their constraint on $\omega'$
(which in their notation is $\kappa$
for both,
their one- and their two-argument kernel)
is bi-equivariance with respect to the stabilizer
of the group action ${G \curvearrowright B}$;
this is \mbox{\citep[Theorem 3.2]{10.5555/3454287.3455107}}.
Specialized to the present example
of the abelian group $G$
and trivial vector bundles over $B$,
bi-equivariance amounts to invariance
with respect to addition of elements in the stabilizer
${\{(k, -k)\}_{k \in \Z} \subset G}$,
i.e.
\begin{equation}
  \label{eq:stabilizerInvariance}
  \omega'(g_1, g_2) \stackrel{!}{=} \omega'(g_1 + k, g_2 - k)
\end{equation}
for all ${g_1 \in \R}$ and ${g_2, k \in \Z}$.
Now independent of our choice for the map
${\theta \colon \supp \kappa \rightarrow G}$,
any filter
${\omega' \colon G \rightarrow \R}$
obtained by the above construction
will not satisfy this equation \eqref{eq:stabilizerInvariance}.

Now suppose that by some other means
we obtained some continuous function
${\omega' \colon G \rightarrow \R}$
satisfying the equation \eqref{eq:stabilizerInvariance}
for all ${g_1 \in \R}$ and ${g_2, k \in \Z}$.
As $G$ is an abelian group acting transitively on ${B = \R}$
we now use the simplified form of a cross-correlation
provided by \eqref{eq:xcorrHatGeneralized}
with respect to the measure
${\mu' \coloneqq \mu}$.
Moreover,
let ${f \colon B \rightarrow \R}$
be a continuous function
and let ${b \in B}$.
Then as far as the first integral of the following calculation is finite,
Fubini's theorem implies
\begin{align*}
  \big(\omega' \hatstar f\big)(b)
  &=
    \int_G \omega'(g) f(g.b) d\mu'(g)
  \\
  &=
    \sum_{g_2 \in \Z~}
    \int_{-\infty}^{\infty}
    \omega'(g_1, g_2) f(g_1 + g_2 + b)
    dg_1
  \\
  &\!\!\stackrel{\eqref{eq:stabilizerInvariance}}{=}
    \sum_{g_2 \in \Z~}
    \int_{-\infty}^{\infty}
    \omega'(g_1 + g_2, g_2 - g_2) f(g_1 + g_2 + b)
    dg_1
  \\
  &=
    \sum_{g_2 \in \Z~}
    \int_{-\infty}^{\infty}
    \omega'(g_1 + g_2, 0) f(g_1 + g_2 + b)
    dg_1
  \\
  &=
    \sum_{g_2 \in \Z~}
    \int_{-\infty}^{\infty}
    \omega'(g_1, 0) f(g_1 + b)
    dg_1
\end{align*}
and hence ${\big(\omega' \hatstar f\big)(b) = 0}$.
So for the group action ${G \curvearrowright B}$,
bi-equivariance of the filter/\enquote{one-argument kernel} $\omega'$
results in vanishing, degenerate, or ill-defined cross-correlations.
With that said,
the present example is ruled out
when assuming compact stabilizers as in
\mbox{\citep[Remark 3.2]{Gerken2023-dm}}.

In summary,
the constraint \eqref{eq:omegaConstraint}
on the filter $\omega$,
which specializes to no constraint on $\omega'$
for the present example,
is lenient enough to accommodate non-compact stabilizers.
Moreover,
while this additional flexibility also results
in more than one filter providing the same integral transform,
it allows the description as a cross-correlation
to inform the shape of the tensor holding
the trainable parameters of the filter $\omega'$
(or $\omega$ in more general settings).

\subsection{Compatible Measures}

When defining cross-correlations
in \cref{definition:xcorr},
we used a family of measures
${\{\mu_b \colon \mathcal{B}(G) \rightarrow [0, \infty]\}_{b \in B}}$
defined on the group $G$
and for integral transforms we use a family of measures
${\{
  \bar{\mu}_b \colon \mathcal{B}(G.b) \rightarrow [0, \infty]
  \}_{b \in B}}$
defined
on the orbits of the action ${G \curvearrowright B}$.
In order to link these two families of measures,
we assume there is a third family
\begin{equation*}
  \{\nu_b \colon \mathcal{B}(G_b) \rightarrow [0, \infty]\}_{b \in B}
\end{equation*}
of left-invariant locally finite Borel measures
on the stabilizers of the action ${G \curvearrowright B}$
with the following two properties.
Closely analogous to our constraints on the family
${\{\mu_b\}_{b \in B}}$
we require that we have
\begin{equation*}
  \nu_{g.b} = c_{g*} \nu_b
\end{equation*}
for all ${g \in G}$ and ${b \in B}$,
where ${c_{g*} \nu_b}$ is the pushforward measure of $\nu_b$
along the conjugation
\begin{equation*}
  c_g \colon G_b \rightarrow G_{g.b}, \,
  h \mapsto ghg^{-1}
\end{equation*}
here as a map between stabilizers.
Secondly,
relating the three families of measures,
we assume we have the equation
\begin{equation}
  \label{eq:fubini}
  \int_G
  f(h) d\mu_b(h)
  =
  \int_{G.b}
  \int_{G_b} f(kh) d\nu_b(h)
  d\bar{\mu}_b(k.b)  
\end{equation}
for all ${b \in B}$ and compactly supported continuous functions
${f \colon G \rightarrow \R}$.
In this equation \eqref{eq:fubini} we view $k.b$ as a pattern
being matched against all $c$ within the domain of integration $G.b$.
More specifically,
as we integrate over ${c \in G.b}$
the free variable $k$ is bound to some ${k \in G}$
such that ${c = k.b}$.
As the measure $\nu_b$ is left-invariant,
the value of the inner integral
\begin{equation*}
  \int_{G_b} f(kh) d\nu_b(h)
\end{equation*}
is independent of the particular choice of $k$
satisfying the equation ${c = k.b}$.

Now in order to construct a filter $\omega$ from a kernel $\kappa$
we will also need a way of associating to a real number ${r \in \R}$
a function ${f \colon G_b \rightarrow \R}$ (where ${b \in B}$)
whose integral
${\int_{G_b} f(h) d\nu_b(h)}$ evaluates to $r$.
If $G_b$ is compact,
then we may define $f$ to be the constant function
evaluating to ${r / \nu_b(G_b)}$.
However,
this construction only works when ${\nu_b(G_b)}$ is finite
and even then,
we might prefer to concentrate the distribution of values
towards the neutral element ${e \in G_b}$
in order to limit the support of the resulting filter $\omega$.
To this end,
we further assume we have a continuous function
\begin{equation*}
  \delta \colon
  \bigsqcup_{b \in B} G_b
  \rightarrow [0, \infty),
\end{equation*}
where
$\displaystyle{
  \bigsqcup_{b \in B} G_b = \{(h, b) \in G \times B \mid h \in G_b\}
}$
inherits the subspace topology from ${G \times B}$,
with the following properties.
For each ${b \in B}$
the partially applied function
${\delta(-, b) \colon G_b \rightarrow [0, \infty)}$
is compactly supported
and normed with respect to $\nu_b$:
\begin{equation}
  \label{eq:deltaNormed}
  \int_{G_b} \delta(h, b) d\nu_b(h) = 1.
\end{equation}
Moreover,
we assume we have the constraint
\begin{equation}
  \label{eq:deltaConstraint}
  \delta\big(ghg^{-1}, g.b\big) = \delta(h, b)
\end{equation}
for all ${g \in G}$, ${b \in B}$, and ${h \in G_b}$.

In most situations,
we may well prefer to choose ${\delta}$ in such a way
that integration of each partially applied function
${\delta(-, b) \colon G_b \rightarrow [0, \infty)}$ with ${b \in B}$
over any Borel set of ${G_b}$
provides an approximation of the dirac measure on $G_b$,
which sends Borel sets containing the neutral element $e$ to $1$
and all other Borel sets to $0$.

\begin{examples}
  In support of these assumptions,
  we provide the following two examples.
  \begin{enumerate}[(i)]
  \item
    In the example discussed in the previous \cref{sec:example},
    the stabilizer of any ${b \in B = \R}$
    is the discrete additive subgroup
    \begin{equation*}
      G_b =
      \{(g_1, g_2) \in \R \times \Z \mid g_1 - g_2 = 0\}
      \eqcolon H
      \cong \Z.
    \end{equation*}
    So we may choose
    ${\nu_b \colon \mathcal{B}(G_b) \rightarrow [0, \infty]}$
    to be the counting measure on ${G_b = H}$
    for all ${b \in B}$.
    With $\{\mu_b\}_{b \in B}$ and $\{\bar{\mu}_b\}_{b \in B}$
    defined as in \cref{sec:example},
    we also have the equation \eqref{eq:fubini}
    for all ${b \in B}$ and compactly support continuous
    ${f \colon G \rightarrow \R}$.
    Finally,
    we may define
    \begin{equation*}
      \delta \colon H \times B \rightarrow [0, \infty), \,
      h \mapsto
      \begin{cases}
        1 & h = (0, 0) \\
        0 & \text{otherwise,}
      \end{cases}
    \end{equation*}
    which is easily seen to satisfy the equations
    \eqref{eq:deltaNormed} and \eqref{eq:deltaConstraint}.
    Using these choices with the general construction that follows
    in \cref{sec:liftKernelToFilter},
    we recover the filter $\omega$
    we described in \cref{sec:example}
    (for the corresponding choice of $\theta$).
  \item
    For a more generic example,
    we assume the action ${G \curvearrowright B}$
    satisfies the \cref{assumption:GInvMeas}.
    In this case,
    which is a generalization of (i),
    the \cref{thm:compatFamMeas}
    provides families of measures
    $\{\mu_b\}_{b \in B}$, 
    $\{\bar{\mu}_b\}_{b \in B}$,
    and
    $\{\nu_b\}_{b \in B}$
    satisfying all of the above constraints (and more)
    as well as a function
    \begin{equation*}
      \psi \colon G \times B \rightarrow [0, \infty)
    \end{equation*}
    with
    ${\psi(-, b) \colon G \rightarrow [0, \infty)}$
    compactly supported for each ${b \in B}$,
    with
    \begin{equation*}
      \psi(h, b) = \psi\big(ghg^{-1}, g.b\big)
      \tag{\ref{eq:psiConstraint} anticipated}
    \end{equation*}
    for all ${g, h \in G}$ and ${b \in B}$,
    and with
    \begin{equation}
      \label{eq:psiRestrictedNormed}
      \int_{G_b}
      \psi(h, b) d\nu_b(h)
      = 1
    \end{equation}
    for all ${b \in B}$.
    Using the function
    ${\psi \colon G \times B \rightarrow [0, \infty)}$
    we define the function
    \begin{equation*}
      \delta \colon \bigsqcup_{b \in B} G_b \rightarrow [0, \infty), \,
      (h, b) \mapsto \psi(h, b)
    \end{equation*}
    as its restriction
    to
    ${\{(h, b) \in G \times B \mid h \in G_b\}}$.
    As ${\psi(-, b) \colon G \rightarrow [0, \infty)}$
    is compactly supported for each ${b \in B}$,
    corresponding partially applied functions
    ${\delta(-, b) \colon G_b \rightarrow [0, \infty)}$
    are compactly supported as well.
    Moreover,
    the equations
    \eqref{eq:deltaNormed} and \eqref{eq:deltaConstraint}
    follow directly from the equations
    \eqref{eq:psiRestrictedNormed} and \eqref{eq:psiConstraint}
    respectively.
  \end{enumerate}
\end{examples}

\subsection{Projection of Filters to Kernels}
\label{sec:projectFilterToKernel}

Before we get to lifting kernels
to filters for cross-correlations,
we provide a construction of the converse:
How an integral transform
can be obtained from a filter for cross-correlations.
To this end,
suppose we have $G$-equivariant real vector bundles
${E \rightarrow B}$ and ${F \rightarrow B}$
and let ${\omega \colon G \times B \rightarrow \mathrm{Hom}(E, F)}$
be a filter as in \cref{sec:xcorrFilter}.
Then we define the kernel
\begin{equation*}
  \kappa \colon \bigsqcup_{b \in B} G.b \rightarrow
  \mathrm{Hom}(E \times B, B \times F)
\end{equation*}
by
\begin{equation}
  \label{eq:projectionToKernel}
  \kappa(k.b, b)(v) \coloneqq
  \int_{G_b} \omega(kh, b)\big(h^{-1}k^{-1}.v\big) d\nu_b(h)
\end{equation}
for all ${b \in B}$, ${k \in G}$, and ${v \in E_{k.b}}$.
As the Borel measure
${\nu_b \colon \mathcal{B}(G_b) \rightarrow [0, \infty]}$
is left-invariant for any ${b \in B}$,
the kernel $\kappa$ is not overdetermined
by these assignments \eqref{eq:projectionToKernel}.

\begin{lemma}
  Let ${g \in G}$, ${b \in B}$, ${c \in G.b}$, and ${v \in E_c}$.
  Then we have the equation
  \begin{equation*}
    g.\kappa(c, b)(v) = \kappa(g.c, g.b)(g.v).
    \tag{\ref{eq:kappaConstraint} revisited}
  \end{equation*}
\end{lemma}

\begin{proof}
  Let ${k \in G}$ with ${c = k.b}$.
  Then we have
  \begin{align*}
    g.\kappa(c, b)(v)
    &=
      g.\kappa(k.b, b)(v)
    \\
    &\!\!\stackrel{\eqref{eq:projectionToKernel}}{=}
      g.\int_{G_b} \omega(kh, b)\big(h^{-1}k^{-1}.v\big) d\nu_b(h)
    \\
    &=
      \int_{G_b} g.\omega(kh, b)\big(h^{-1}k^{-1}.v\big) d\nu_b(h)
    \\
    &\!\stackrel{\eqref{eq:omegaConstraint}}{=}
      \int_{G_b}
      \omega\big(gkhg^{-1}, g.b\big)\big(gh^{-1}k^{-1}.v\big)
      d\nu_b(h)
    \\
    &=
      \int_{G_b}
      \omega\big(gkg^{-1}ghg^{-1}, g.b\big)\big(gh^{-1}g^{-1}gk^{-1}.v\big)
      d\nu_{b}(h)
    \\
    &=
      \int_{G_b}
      \omega\big(gkg^{-1}h, g.b\big)\big(h^{-1}gk^{-1}.v\big)
      dc_{g*} \nu_{b}(h)
    \\
    &=
      \int_{G_b}
      \omega\big(gkg^{-1}h, g.b\big)\big(h^{-1}gk^{-1}g^{-1}g.v\big)
      d\nu_{g.b}(h)
    \\
    &\!\!\stackrel{\eqref{eq:projectionToKernel}}{=}
      \kappa\big(gkg^{-1}g.b, g.b\big)(g.v)
    \\[1ex]
    &=
      \kappa(gk.b, g.b)(g.v)
    \\[1ex]
    &=
      \kappa(g.c, g.b)(g.v).
      \qedhere
  \end{align*}
\end{proof}

We also note that,
as the map
${G \rightarrow B, \, k \mapsto k.b}$
is continuous
and as
${\omega(-, b) \colon G \rightarrow \mathrm{Hom}(E_b, F_b)}$
has compact support,
the support of the partially applied map
${\kappa(-, b) \colon G.b \rightarrow \mathrm{Hom}(E, F_b)}$
is compact as well for all ${b \in B}$.
  
\begin{theorem}
  For any continuous section ${f \in \Gamma(E)}$ and any ${b \in B}$
  we have
  \begin{equation*}
    T_{\kappa}(f)(b) = \big(\omega \star \tilde{f}\big)(e, b),
  \end{equation*}
  where
  ${\tilde{f} \colon G \times B \rightarrow E}$
  is the Mackey section associated to $f$
  in the sense of \cref{definition:mackeySection}.
\end{theorem}

\begin{proof}
  We have
  \begin{align*}
    T_{\kappa}(f)(b)
    &=
      \int_{G.b} \kappa(c, b)(f(c)) d\bar{\mu}_b(c)
    \\
    &=
      \int_{G.b} \kappa(k.b, b)(f(k.b)) d\bar{\mu}_b(k.b)
    \\
    &\!\!\stackrel{\eqref{eq:projectionToKernel}}{=}
      \int_{G.b}
      \int_{G_b} \omega(kh, b)\big(h^{-1}k^{-1}.f(k.b)\big) d\nu_b(h)
      d\bar{\mu}_b(k.b)
    \\
    &=
      \int_{G.b}
      \int_{G_b} \omega(kh, b)\big((kh)^{-1}.f(kh.b)\big) d\nu_b(h)
      d\bar{\mu}_b(k.b)
    \\
    &\!\stackrel{\eqref{eq:definitionMackeySection}}{=}
      \int_{G.b}
      \int_{G_b} \omega(kh, b)\big(\tilde{f}(kh, b)\big) d\nu_b(h)
      d\bar{\mu}_b(k.b)
    \\
    &=
      \int_{G}
      \omega(h, b)\big(\tilde{f}(h, b)\big) d\mu_b(h)
    \\
    &=
      \big(\omega \star \tilde{f}\big)(e, b).
      \qedhere
  \end{align*}
\end{proof}

\subsection{Lifting Kernels to Filters}
\label{sec:liftKernelToFilter}

We now provide a converse
to the construction of the previous \cref{sec:projectFilterToKernel}.
To this end,
suppose we have $G$-equivariant real vector bundles
${E \rightarrow B}$ and ${F \rightarrow B}$
and let
\[\kappa \colon
  \bigsqcup_{b \in B} G.b \rightarrow
  \mathrm{Hom}(E \times B, B \times F)\]
be a kernel as in \cref{sec:orbitwiseIntegralTrafos}.
As a partially applied cross-correlation
${\omega \star - \colon M(E) \rightarrow M(F)}$
is $G$-equivariant by \cref{thm:xcorrEquivariance}
for a filter $\omega$ satisfying the constraint
\eqref{eq:omegaConstraint},
the orbitwise integral transform
${T_{\kappa} \colon \Gamma(E) \rightarrow \Gamma'(F)}$
is necessarily $G$-equivariant
if it can be written in terms of such cross-correlations.
Moreover,
under the mild tameness assumptions
of \cref{thm:equivToConstraint},
the integral transform $T_{\kappa}$ is $G$-equivariant
iff its kernel $\kappa$ satisfies the constraint
\begin{equation*}
  g.\kappa(c, b)(v) = \kappa(g.c, g.b)(g.v)
  \tag{\ref{eq:kappaConstraint} revisited}
\end{equation*}
for all ${g \in G}$, ${b \in B}$, ${c \in G.b}$, and ${v \in E_c}$,
which we assume to be satisfied from this point forward.

As we have seen already with the example of \cref{sec:example},
lifting a kernel $\kappa$ to a filter $\omega$
depends on a choice of the following.
We need some closed subset
${R \subseteq \displaystyle{\bigsqcup_{b \in B}} G.b}$
that contains the support of ${\kappa}$ and
that is $G$-invariant with respect to the diagonal action
${G \curvearrowright B \times B}$.
Moreover,
we need to choose some continuous map
\[\theta \colon R \rightarrow G\]
subject to the constraint
\begin{equation*}
  c = \theta(c, b).b
  \tag{\ref{eq:thetaLift} revisited}
\end{equation*}
for all ${(c, b) \in R}$.
Now in order for our construction
to promote the given constraint \eqref{eq:kappaConstraint}
for the kernel $\kappa$
to the constraint \eqref{eq:omegaConstraint},
which we require from the filter $\omega$,
we need to impose an additional constraint on $\theta$.
More specifically,
for any ${g \in G}$ and ${(c, b) \in R}$
we impose the equation
\begin{equation}
  \label{eq:thetaConstraint}
  g \theta(c, b) = \theta(g.c, g.b) g.
\end{equation}
While we make no use of the so called
\emph{category of elements}
associated to the group action ${G \curvearrowright B}$
(as a set-valued functor on $G$),
it does provide the illustration
\begin{equation*}
  \begin{tikzcd}[ampersand replacement=\&,column sep=11ex,row sep=5ex]
    b
    \arrow[r, "{\theta(c, b)}"]
    \arrow[d, "g"']
    \&
    c
    \arrow[d, "g"]
    \\
    g.b
    \arrow[r, "{\theta(g.c, g.b)}"']
    \&
    g.c
  \end{tikzcd}
\end{equation*}
of this additional constraint \eqref{eq:thetaConstraint}.

\begin{remark}
  \label{remark:theta}
  We add some comments related to the map
  ${\theta \colon R \rightarrow G}$.
  \begin{enumerate}[(i)]
  \item
    The additional constraint \eqref{eq:thetaConstraint}
    is equivalent to the map
    ${\theta \colon R \rightarrow G}$
    being $G$-equivariant
    with respect to the diagonal action on
    ${R \subseteq B \times B}$
    and the action by conjugation on $G$.
  \item
    Let ${b \in B}$,
    ${S \coloneqq \{c \in G.b \mid (c, b) \in R\}}$,
    and let ${E |_S \coloneqq \bigcup_{c \in S} E_c}$
    be the restriction of $E$ to $S$.
    Then the partially applied map
    ${\theta(-, b) \colon S \rightarrow G}$
    provides the trivialization
    \begin{equation*}
      \begin{tikzcd}[ampersand replacement=\&,row sep=0.5ex]
        S \times E_b
        \arrow[r, "\cong"]
        \&
        E |_S
        \\
        (c, v)
        \arrow[r, mapsto]
        \&
        \theta(c, b).v
      \end{tikzcd}
    \end{equation*}
    of the restricted vector bundle
    ${E |_S \rightarrow S}$.
    As we require the support of the partially applied kernel
    ${\kappa(-, b)}$ to be contained in $S$,
    the present construction can only be applied
    when the vector bundle
    ${E \rightarrow B}$
    is trivializable above the receptive field
    of $b$.
    In order to circumvent this limitation,
    we generalize the present construction
    in \cref{sec:generalizationLargeReceptiveFields}.
  \end{enumerate}
\end{remark}

As we collected all the necessary ingredients,
we now define the filter
${\omega \colon G \times B \rightarrow \mathrm{Hom}(E, F)}$
by setting
\begin{equation}
  \label{eq:omegaDefinition}
  \omega(h, b)(v) \coloneqq
  \begin{cases}
    \delta\big(\theta(h.b, b)^{-1}h, b\big)\kappa(h.b, b)(h.v)
    & (h.b, b) \in R
    \\
    0
    & (h.b, b) \notin R
  \end{cases}
\end{equation}
for all ${h \in G}$, ${b \in B}$, and ${v \in E_b}$.

\begin{lemma}
  \label{thm:omegaCompactSupport}
  For ${b \in B}$
  the partially applied map
  ${\omega(-, b) \colon G \rightarrow \mathrm{Hom}(E_b, F_b)}$
  has compact support.
\end{lemma}

\begin{proof}
  Let ${b \in B}$,
  let
  ${S \coloneqq \{c \in G.b \mid (c, b) \in R\}}$,
  let
  \begin{equation*}
    \varphi \colon S \times G_b \rightarrow G,\,
    (c, h) \mapsto \theta(c, b)h,
  \end{equation*}
  and let
  \begin{equation*}
    C \coloneqq
    (\supp \kappa(-, b)) \times (\supp \delta(-, b)) \subseteq S \times G_b.
  \end{equation*}
  As $C$ is compact by our assumptions
  on the kernel $\kappa$ and the function $\delta$
  and as $\varphi$ is continuous,
  it suffices to show the inclusion
  ${\supp \omega(-, b) \subseteq \varphi(C)}$.
  To this end,
  let ${h \in \supp \omega(-, b)}$.
  Then we have
  ${(h.b, b) \in \supp \kappa(-, b)}$
  and
  ${k \coloneqq \theta(h.b, b)^{-1} h \in \supp \delta(-, b)}$.
  As a result,
  we have
  ${(h.b, k) \in C}$
  and
  \begin{equation*}
    h =
    \theta(h.b, b) \theta(h.b, b)^{-1} h =
    \theta(h.b, b)k =
    \varphi(h.b, k) \in \varphi(C).
    \qedhere
  \end{equation*}
\end{proof}

\begin{lemma}
  \label{thm:omegaConstraint}
  Let ${g, h \in G}$, ${b \in B}$, and ${v \in E_b}$.
  Then
  we have the equation
  \begin{equation*}
      \omega\big(ghg^{-1}, g.b\big)(g.v) = g.\omega(h, b)(v).
      \tag{\ref{eq:omegaConstraint} revisited}
  \end{equation*}
\end{lemma}

\begin{proof}
  As the subset
  $\displaystyle{
    R \subseteq \bigsqcup_{b' \in B} G.b'
  }$
  is assumed or chosen to be $G$-invariant
  with respect to the diagonal action
  ${G \curvearrowright B \times B}$,
  we have
  ${(h.b, b) \in R}$ iff we have
  ${\big(ghg^{-1}g.b, g.b\big) = (gh.b, g.b) \in R}$.
  Thus,
  if ${(h.b, b) \in R}$,
  then we obtain the equation
  \begin{align*}
    \omega\big(ghg^{-1}, g.b\big)(g.v)
    &\stackrel{\eqref{eq:omegaDefinition}}{=}
      \delta\big(\theta(gh.b, g.b)^{-1}ghg^{-1}, g.b\big)
      \kappa(gh.b, g.b)(gh.v)
    \\
    &\stackrel{\eqref{eq:kappaConstraint}}{=}
      \delta\big(\theta(gh.b, g.b)^{-1}ghg^{-1}, g.b\big)
      g.\kappa(h.b, b)(h.v)
    \\
    &\stackrel{\eqref{eq:thetaConstraint}}{=}
      \delta\big(g\theta(h.b, b)^{-1}hg^{-1}, g.b\big)
      g.\kappa(h.b, b)(h.v)
    \\
    &\stackrel{\eqref{eq:deltaConstraint}}{=}
      \delta\big(\theta(h.b, b)^{-1}h, b\big)
      g.\kappa(h.b, b)(h.v)      
    \\
    &~=\,
      g.\big(\delta\big(\theta(h.b, b)^{-1}h, b\big)\kappa(h.b, b)(h.v)\big)
    \\
    &\stackrel{\eqref{eq:omegaDefinition}}{=}
      g.\omega(h, b)(v)
  \end{align*}
  and if ${(h.b, b) \notin R}$,
  then we have
  \begin{equation*}
    \omega\big(ghg^{-1}, g.b\big)(g.v) = 0 = g.\omega(h, b)(v).
    \qedhere
  \end{equation*}
\end{proof}

In view of the preceding Lemmas \ref{thm:omegaCompactSupport}
and \ref{thm:omegaConstraint},
the map ${\omega \colon G \times B \rightarrow \mathrm{Hom}(E, F)}$
can now be used as a filter
in cross-correlations
in the sense of \cref{definition:xcorr}.

\begin{theorem}
  \label{thm:omegaIsLift}
  Let
  ${f \in \Gamma(E)}$
  be a continuous section
  and let
  \begin{equation*}
    \tilde{f} \colon G \times B \rightarrow E, \,
    (h, b) \mapsto h^{-1}.f(h.b)
  \end{equation*}
  be the corresponding Mackey section
  in the sense of \cref{definition:mackeySection}.
  Then we have
  \begin{equation*}
    \big(\omega \star \tilde{f}\big)(e, b)
    =
    T_{\kappa}(f)(b)
  \end{equation*}
  for any ${b \in B}$.
\end{theorem}

\begin{proof}
  Let ${b \in B}$,
  ${S \coloneqq \{c \in G.b \mid (c, b) \in R\}}$,
  and ${\tilde{S} \coloneqq \{h \in G \mid h.b \in S\}}$.
  For any ${h \in \tilde{S}}$ we have
  \begin{equation}
    \label{eq:hCanceling}
    \begin{split}
      \omega(h, b)\big(\tilde{f}(h, b)\big)
      &\stackrel{\eqref{eq:omegaDefinition}}{=}
        \delta\big(\theta(h.b, b)^{-1}h,b\big)
        \kappa(h.b, b)\big(hh^{-1}.f(h.b)\big)
      \\
      &~=
        \delta\big(\theta(h.b, b)^{-1}h,b\big)
        \kappa(h.b, b)(f(h.b)).      
    \end{split}
  \end{equation}
  Hence,
  for any ${c \in S}$ and ${h \in G_b}$ we obtain the equation
  \begin{equation}
    \label{eq:omega_delta}
    \begin{split}
      \omega(\theta(c, b)h, b)\big(\tilde{f}(\theta(c, b)h, b)\big)
      \!\!\!\!\!\!\!\!\!\!\!\!\!\!\!\!\!\!\!\!\!\!\!\!\!\!\!\!\!\!
      \!\!\!\!\!\!\!\!\!\!\!\!\!\!\!\!\!
      \\
      &\stackrel{\eqref{eq:hCanceling}}{=}
        \delta\big(\theta(\theta(c, b)h.b, b)^{-1} \theta(c, b)h, b\big)
        \kappa(\theta(c, b)h.b, b)(f(\theta(c, b)h.b))
      \\
      &\stackrel{\eqref{eq:thetaLift}}{=}
        \delta\big(\theta(c, b)^{-1} \theta(c, b)h, b\big)
        \kappa(c, b)(f(c))
      \\
      &~=\,
        \delta(h, b)
        \kappa(c, b)(f(c)).      
    \end{split}
  \end{equation}
  As an end result we obtain
  \begin{align*}
    \big(\omega \star \tilde{f}\big)(e, b)
    &~=\,
      \int_G \omega(h, b)\big(\tilde{f}(h, b)\big)d\mu_b(h)
    \\
    &~=\,
      \int_{\tilde{S}} \omega(h, b)\big(\tilde{f}(h, b)\big)d\mu_b(h)
    \\
    &\stackrel{\eqref{eq:fubini}}{=}
      \int_S \int_{G_b}
      \omega(kh, b)\big(\tilde{f}(kh, b)\big)
      d\nu_b(h) d\bar{\mu}_b(k.b)
    \\
    &\stackrel{\eqref{eq:thetaLift}}{=}
      \int_S \int_{G_b}
      \omega(\theta(c, b)h, b)\big(\tilde{f}(\theta(c, b)h, b)\big)
      d\nu_b(h) d\bar{\mu}_b(c)
    \\
    &\stackrel{\eqref{eq:omega_delta}}{=}
      \int_S \int_{G_b}
      \delta(h, b)
      \kappa(c, b)(f(c))
      d\nu_b(h) d\bar{\mu}_b(c)    
    \\
    &~=\,
      \int_S \int_{G_b}
      \delta(h, b)
      d\nu_b(h)
      \kappa(c, b)(f(c))
      d\bar{\mu}_b(c)    
    \\
    &\stackrel{\eqref{eq:deltaNormed}}{=}
      \int_S
      \kappa(c, b)(f(c))
      d\bar{\mu}_b(c)
    \\
    &~=\,
      T_{\kappa}(f)(b).
      \qedhere
  \end{align*}
\end{proof}

\begin{corollary}
  For any continuous section
  ${f \in \Gamma(E)}$
  the section
  ${T_{\kappa}(f) \in \Gamma'(F)}$
  is continuous as well
  and hence a posteriori a vector in $\Gamma(F)$.
\end{corollary}

\begin{proof}
  Let ${f \in \Gamma(E)}$.
  By the preceding \cref{thm:omegaIsLift}
  we have
  ${T_{\kappa}(f) = \big(\omega \star \tilde{f}\big)(e, -)}$,
  which is continuous.
\end{proof}

In the case of a transitive action
with compact stabilizers
by a unimodular group $G$,
\citet{Aronsson2022-ca}
has shown that
any $G$-equivariant transformation of vector bundle sections
satisfying some additional tameness assumption
can be obtained from cross-correlations.
His construction is more abstract
than the above construction of $\omega$
from $\kappa$ and $\theta$,
and it is not clear whether the resulting filter
(there referred to as a kernel and denoted by $\kappa$)
satisfies any particular constraints
(such as bi-equivariance or a constraint similar to ours).

\subsubsection{Generalization to Large Receptive Fields}
\label{sec:generalizationLargeReceptiveFields}

The preceding construction
lifting the filter $\kappa$
to a filter $\omega$
depended on the choice of a
$G$-invariant subset
${R \subseteq \displaystyle{\bigsqcup_{b \in B}} G.b}$
supporting the kernel $\kappa$
and a continuous map
${\theta \colon R \rightarrow G}$
subject to the constraint
\begin{equation*}
  c = \theta(c, b).b
  \tag{\ref{eq:thetaLift} revisited}
\end{equation*}
for all ${(c, b) \in R}$,
which is also reasonable considering the example of \cref{sec:example}.
However,
in view of \cref{remark:theta}.(ii),
the existence of such a map
${\theta \colon R \rightarrow G}$
also entails that the vector bundle
${E \rightarrow B}$
is trivializable above the receptive field
of any ${b \in B}$.
For most applications,
this is likely no serious limitation,
as the concept of a convolutional neural network
is also informed by the idea
that the receptive field of each node
will be localized around this node.
So it is not a far stretch to assume
that the receptive field of each ${b \in B}$
will be sufficiently local
for ${E \rightarrow B}$ to be trivializable
above the receptive field of $b$.
With that said,
it is still a reasonable question,
whether the concept of a group convolutional layer
can be pushed beyond vector bundles
that are trivializable above each receptive field.
And indeed,
as with many other local constructions,
such can be done using a sufficiently fine partition of unity.

We now assume
there is a locally finite
partition of unity
\[\{\varphi_i \colon
  \textstyle{\bigsqcup_{b \in B}} G.b \rightarrow [0, 1]\}_{i \in I}
\]
on $\displaystyle{\bigsqcup_{b \in B} G.b}$
that is $G$-invariant
with respect to the diagonal action
${G \curvearrowright B \times B}$.
Moreover,
let ${U_i \coloneqq \varphi_i^{-1}(0, 1]}$
for ${i \in I}$
and suppose we have a family
\[\{\theta_i \colon \overline{U}_i \rightarrow G\}_{i \in I}\]
of continuous maps
subject to the constraints
\begin{align}
  \label{eq:thetaFamLift}
  c & = \theta_i(c, b).b
  \\
  \text{and} \quad
  \label{eq:thetaFamConstraint}
  g \theta_i(c, b) & = \theta_i(g.c, g.b) g
\end{align}
for all ${g \in G}$, ${i \in I}$, and ${(c, b) \in \overline{U}_i}$,
where $\overline{U}_i$ is the closure of $U_i$
in $\displaystyle{\bigsqcup_{b \in B} G.b}$.

Also note,
while our previous choice of a supporting region $R$
and a map ${\theta \colon R \rightarrow G}$
had to take the support of the kernel $\kappa$ into consideration,
the partition of unity
${\{\varphi_i\}_{i \in I}}$ and the family of maps
${\{\theta_i\}_{i \in I}}$
can be chosen independently from the kernel $\kappa$.

Now let ${i \in I}$ and
\begin{equation}
  \label{eq:localizedKernel}
  \kappa_i \colon \bigsqcup_{b \in B} G.b \rightarrow
  \mathrm{Hom}(E \times B, B \times F),\,
  (c, b) \mapsto
  \varphi_i (c, b) \kappa(c, b).
\end{equation}
%
%
%
As
${\varphi_i \colon
  \displaystyle{\bigsqcup_{b \in B} G.b} \rightarrow [0, 1]}$
is $G$-invariant
with respect to the diagonal action ${G \curvearrowright B \times B}$,
the constraint \eqref{eq:kappaConstraint}
on the kernel $\kappa$
is inherited by this new kernel $\kappa_i$,
i.e. we have
\begin{equation}
  \label{eq:kappaFamConstraint}
  g.\kappa_i(c, b)(v) = \kappa_i(g.c, g.b)(g.v)
\end{equation}
for all ${g \in G}$, ${b \in B}$, ${c \in G.b}$, and ${v \in E_c}$.
Thus,
we have the $G$-equivariant integral transform
\[T_i \coloneqq
  T_{\kappa_i} \colon \Gamma(E) \rightarrow \Gamma'(F), \,
  f \mapsto T_{\kappa_i}(f)\]
by \cref{thm:kappaConstraintImpliesEquivariance}.
As we have
${\supp \kappa_i \subseteq \overline{U}_i}$
by construction of $\kappa_i$
and as we have the map
${\theta_i \colon \overline{U}_i \rightarrow G}$
as well as the constraints
\eqref{eq:kappaFamConstraint},
\eqref{eq:thetaFamLift}, and
\eqref{eq:thetaFamConstraint},
we are now back in the special case,
where we can lift the kernel $\kappa_i$
to a filter
${\omega_i \colon G \times B \rightarrow \mathrm{Hom}(E, F)}$
by setting
\begin{equation}
  \label{eq:omegaFamDefinition}
  \omega_i(h, b)(v) \coloneqq
  \begin{cases}
    \delta\big(\theta_i(h.b, b)^{-1}h, b\big)\kappa_i(h.b, b)(h.v)
    & (h.b, b) \in \overline{U}_i
    \\
    0
    & (h.b, b) \notin \overline{U}_i
  \end{cases}
\end{equation}
for all ${h \in G}$, ${b \in B}$, and ${v \in E_b}$
in close analogy to \eqref{eq:omegaDefinition}.

\begin{lemma}
  \label{thm:omegaFamCompactSupport}
  For ${i \in I}$ and ${b \in B}$
  the partially applied map
  ${\omega_i(-, b) \colon G \rightarrow \mathrm{Hom}(E_b, F_b)}$
  has compact support.
\end{lemma}

\begin{lemma}
  \label{thm:omegaFamConstraint}
  Let ${i \in I}$, ${g, h \in G}$, ${b \in B}$, and ${v \in E_b}$.
  Then
  we have the equation
  \begin{equation}
    \label{eq:omegaFamConstraint}
      \omega_i\big(ghg^{-1}, g.b\big)(g.v) = g.\omega_i(h, b)(v).
  \end{equation}
\end{lemma}

\begin{proof}[Proof of Lemmas
  \ref{thm:omegaFamCompactSupport} and
  \ref{thm:omegaFamConstraint}]
  These are Lemmas \ref{thm:omegaCompactSupport}
  and \ref{thm:omegaConstraint}
  now just reformulated for the herein defined filter
  ${\omega_i \colon G \times B \rightarrow \mathrm{Hom}(E, F)}$.  
\end{proof}

In view of the preceding Lemmas \ref{thm:omegaFamCompactSupport}
and \ref{thm:omegaFamConstraint},
the map ${\omega_i \colon G \times B \rightarrow \mathrm{Hom}(E, F)}$
can now be used as a filter
in cross-correlations
in the sense of \cref{definition:xcorr}
for each ${i \in I}$.

\begin{theorem}
  \label{thm:omegaFamIsLift}
  Let
  ${i \in I}$,
  let
  ${f \in \Gamma(E)}$
  be a continuous section,
  and let
  \begin{equation*}
    \tilde{f} \colon G \times B \rightarrow E, \,
    (h, b) \mapsto h^{-1}.f(h.b)
  \end{equation*}
  be the corresponding Mackey section
  in the sense of \cref{definition:mackeySection}.
  Then we have
  \begin{equation}
    \label{eq:omegaFamIsLift}
    \big(\omega_i \star \tilde{f}\big)(e, b)
    =
    T_{i}(f)(b)
  \end{equation}
  for any ${b \in B}$.
\end{theorem}

\begin{proof}
  This is \cref{thm:omegaIsLift}
  now just reformulated for the kernel $\kappa_i$,
  the integral transform
  ${T_i =
    T_{\kappa_i} \colon \Gamma(E) \rightarrow \Gamma'(F)}$,
  and the filter
  ${\omega_i \colon G \times B \rightarrow \mathrm{Hom}(E, F)}$.  
\end{proof}

With the preceding \cref{thm:omegaFamIsLift}
we now managed to lift each localized kernel $\kappa_i$
to a filter $\omega_i$
for ${i \in I}$.
In order to lift the kernel $\kappa$ itself,
we intend to use the sum of the family of filters
${\{\omega_i\}_{i \in I}}$:
\begin{equation}
  \label{eq:omegaSum}
  \omega \colon
  G \times B \rightarrow \mathrm{Hom}(E, F),\,
  (h, b) \mapsto
  \sum_{i \in I} \omega_i(h, b).
\end{equation}
Note however,
as the indexing set $I$ may be infinite,
the sum \eqref{eq:omegaSum} may not be well-defined a priori.
In order to prove that \eqref{eq:omegaSum} is well-defined
and also in order to show that it does indeed
lift the kernel $\kappa$,
we now use the locally finite open cover
${\{U_i\}_{i \in I}}$ of
$\displaystyle{\bigsqcup_{b \in B} G.b}$,
to obtain a locally finite open cover
of the orbit ${G.b}$ for each ${b \in B}$ as follows.
Let
\begin{equation*}
  U^b_i \coloneqq
  \{c \in G.b \mid (c, b) \in U_i\}
\end{equation*}
for each ${i \in I}$ and ${b \in B}$.
Then the family
${\{U^b_i\}_{i \in I}}$
is a locally finite open cover of $G.b$
for each ${b \in B}$.

\begin{lemma}
  \label{thm:finiteCoverOnOrbit}
  For each ${b \in B}$
  the set
  \[I_b \coloneqq
    \{i \in I \mid U^b_i \cap \supp \kappa(-, b) \neq \emptyset\}\]
  is finite.
\end{lemma}

\begin{proof}
  This follows from \cref{thm:locallyFiniteCoverCompact}
  and our assumption that the partially applied kernel
  ${\kappa(-, b) \colon G.b \rightarrow \mathrm{Hom}(E, F_b)}$
  is compactly supported.
\end{proof}

Now let ${b \in B}$.
For ${i \in I}$ the partially applied filter
${\omega_i(-, b)}$ 
vanishes iff the partially applied kernel
${\kappa_i(-, b) \colon G.b \rightarrow \mathrm{Hom}(E, F_b)}$
vanishes.
Thus,
by the preceding \cref{thm:finiteCoverOnOrbit}
there are only finitely many ${i \in I}$
with
${\omega_i(-, b) \colon G \rightarrow \mathrm{Hom}(E, F)}$
non-vanishing.
As a result
\begin{equation*}
  \omega(-, b) = \sum_{i \in I} \omega_i(-, b) \colon
  G \rightarrow \mathrm{Hom}(E, F)
\end{equation*}
is a finite sum of compactly supported maps
${G \rightarrow \mathrm{Hom}(E, F)}$,
which implies the following.

\begin{lemma}
  \label{thm:omegaSumCompactSupport}
  For ${b \in B}$
  the partially applied map
  ${\omega(-, b) \colon G \rightarrow \mathrm{Hom}(E_b, F_b)}$
  has compact support.
\end{lemma}

\begin{lemma}
  \label{thm:omegaSumConstraint}
  Let ${g, h \in G}$, ${b \in B}$, and ${v \in E_b}$.
  Then
  we have the equation
  \begin{equation*}
      \omega\big(ghg^{-1}, g.b\big)(g.v) = g.\omega(h, b)(v).
      \tag{\ref{eq:omegaConstraint} revisited}
  \end{equation*}
\end{lemma}

\begin{proof}
  We have
  \begin{align*}
    \omega\big(ghg^{-1}, g.b\big)(g.v)
    & \stackrel{\eqref{eq:omegaSum}}{=}
      \sum_{i \in I} \omega_i\big(ghg^{-1}, g.b\big)(g.v)
    \\
    & \stackrel{\eqref{eq:omegaFamConstraint}}{=}
      \sum_{i \in I} g.\omega_i(h, b)(v)
    \\
    & ~=
      g.\big(\textstyle{\sum_{i \in I}} \omega_i(h, b)(v)\big)
    \\
    & \stackrel{\eqref{eq:omegaSum}}{=}
      g.\omega(h, b)(v).
      \qedhere
  \end{align*}
\end{proof}

In view of the preceding Lemmas \ref{thm:omegaSumCompactSupport}
and \ref{thm:omegaSumConstraint},
the map ${\omega \colon G \times B \rightarrow \mathrm{Hom}(E, F)}$
can now be used as a filter
in cross-correlations
in the sense of \cref{definition:xcorr}.

\begin{theorem}
  \label{thm:omegaSumIsLift}
  Let
  ${f \in \Gamma(E)}$
  be a continuous section
  and let
  \begin{equation*}
    \tilde{f} \colon G \times B \rightarrow E, \,
    (h, b) \mapsto h^{-1}.f(h.b)
  \end{equation*}
  be the corresponding Mackey section
  in the sense of \cref{definition:mackeySection}.
  Then we have
  \begin{equation*}
    \big(\omega \star \tilde{f}\big)(e, b)
    =
    T_{\kappa}(f)(b)
  \end{equation*}
  for any ${b \in B}$.
\end{theorem}

\begin{proof}
  Let ${b \in B}$.
  Then we have
  \begin{align*}
    \big(\omega \star \tilde{f}\big)(e, b)
    & \stackrel{\eqref{eq:xcorr}}{=}
      \int_G \omega(h, b)\big(\tilde{f}(h, b)\big) d\mu_b(h)
    \\
    & \stackrel{\eqref{eq:omegaSum}}{=}
      \int_G \sum_{i \in I} \omega_i(h, b)\big(\tilde{f}(h, b)\big) d\mu_b(h)
    \\
    & ~=
      \sum_{i \in I} 
      \int_G \omega_i(h, b)\big(\tilde{f}(h, b)\big) d\mu_b(h)
    \\
    & \stackrel{\eqref{eq:xcorr}}{=}
      \sum_{i \in I}
      \big(\omega_i \star \tilde{f}\big)(e, b)
    \\
    & \stackrel{\eqref{eq:omegaFamIsLift}}{=}
      \sum_{i \in I}
      T_{i}(f)(b)
    \\
    & \stackrel{\eqref{eq:orbitwiseIntegralTrafo}}{=}
      \sum_{i \in I}
      \int_{G.b} \kappa_i(c, b)(f(c)) d\bar{\mu}_b(c)
    \\
    & ~=
      \int_{G.b}
      \sum_{i \in I}
      \kappa_i(c, b)(f(c)) d\bar{\mu}_b(c)
    \\
    & \stackrel{\eqref{eq:localizedKernel}}{=}
      \int_{G.b}
      \sum_{i \in I}
      \varphi_i (c, b) \kappa(c, b)(f(c)) d\bar{\mu}_b(c)
    \\
    & ~=
      \int_{G.b}
      \kappa(c, b)(f(c)) d\bar{\mu}_b(c)
    \\
    & \stackrel{\eqref{eq:orbitwiseIntegralTrafo}}{=}
      T_{\kappa}(f)(b).
  \end{align*}
  Here the second but last equality follows from the family
  ${\{\varphi_i\}_{i \in I}}$ being a partition of unity.
\end{proof}



\section{Other Integral Transforms}
\label{sec:otherIntegralTrafos}

In the previous \cref{sec:integralTrafosAsXCorr}
we related cross-correlations
to orbitwise integral transforms.
Yet,
in the literature more general integral transforms
are considered as layers of group convolutional neural networks as well.
In this \cref{sec:otherIntegralTrafos}
we discuss the trade-offs between
group cross correlations in the sense of \cref{definition:xcorr}
and more general integral transforms.

To this end,
let $B$ and $C$ be Hausdorff $G$-spaces
with $B$ paracompact
and let
${E \rightarrow B}$ and ${F \rightarrow C}$
be $G$-equivariant vector bundles.
Moreover,
suppose we have a continuous compactly supported section
${\kappa \in \Gamma_c(\mathrm{Hom}(E \times C, B \times F))}$
of the homomorphism bundle
${\mathrm{Hom}(E \times C, B \times F) \rightarrow B \times C}$
as well as a $G$-invariant locally finite strictly positive Borel measure
${\bar{\mu} \colon \mathcal{B}(B) \rightarrow [0, \infty]}$.
By reasoning similar to \cref{sec:equivIntegralTrafo},
we obtain that the associated integral transform
\begin{equation*}
  T_{\kappa} \colon \Gamma(E) \rightarrow \Gamma(F),\,
  f \mapsto T_{\kappa}(f),
\end{equation*}
where
\begin{equation*}
  T_{\kappa}(f) \colon C \rightarrow F,\,
  c \mapsto \int_B \kappa(b, c)(f(b)) d\bar{\mu}(b),
\end{equation*}
is $G$-equivariant
iff we have
\begin{equation}
  \label{eq:kappaConstraintMod}
  g.\kappa(b, c)(v) = \kappa(g.b, g.c)(g.v)
\end{equation}
for all ${g \in G}$, ${b \in B}$, ${c \in C}$, and ${v \in E_b}$;
see also \citep[Section 4.2]{Gerken2023-dm}.

For similar reasons as provided in \cref{sec:contraintsInvariances},
the kernel $\kappa$ is fully determined
by the family of partially applied maps
\begin{equation*}
  \{\kappa(-, c) \colon B \rightarrow \mathrm{Hom}(E, F_c)\}_{c \in D}
\end{equation*}
for some fundamental domain ${D \subseteq C}$
(and similarly for a fundamental domain of $B$).
For each ${c \in D}$ the partially applied map
${\kappa(-, c) \colon B \rightarrow \mathrm{Hom}(E, F_c)}$
is then subject to the constraint
\begin{equation*}
  g.\kappa(b, c)(v) = \kappa(g.b, c)(g.v)
\end{equation*}
for all ${g \in G_c}$, ${b \in B}$, and ${v \in E_b}$.
Now in some literature
(in particular when ${G \curvearrowright C}$ is transitive
so a single point suffices for a fundamental domain)
an integral transform is also considered a \emph{convolution}
when expressed in terms of such partially applied maps;
see for example \citep[Section 4.3]{Gerken2023-dm}.

There is a fundamental difference however.
The partially applied map
${\kappa(-, c) \colon B \rightarrow \mathrm{Hom}(E, F_c)}$
for each ${c \in D}$
still is a section to the potentially twisted vector bundle
${\mathrm{Hom}(E, F_c) \rightarrow B}$.
Whereas in the setting of \cref{sec:xcorrFilter},
the partially applied map
${\omega(-, b) \colon G \rightarrow
  \mathrm{Hom}(E_b, F_b)}$
for some ${b \in B}$
is just a vector-valued function.
This is also illustrated by the diagram \eqref{eq:omegaLift}
where $\omega$ is pictured as a lift to a vector bundle over $B$,
while $\kappa$ is a section to a vector bundle over the product
${B \times C}$.
So in the generality considered in this paper,
there are equivariant integral transforms
that cannot be expressed in terms of cross correlations
with the same benefits as \cref{definition:xcorr}.


\section*{Funding}

This research has been supported by EPSRC grant EP/Y028872/1,
Mathematical Foundations of Intelligence:
An \enquote{Erlangen Programme} for AI.

\bibliography{bib/aronsson--2022.bib,bib/bronstein--et-al--2021.bib,bib/acm_3454287.3455107.bib,bib/pmlr-v48-cohenc16.bib,bib/pmlr-v80-kondor18a.bib,bib/tornier--2020.bib,bib/hatcher--2003.bib,bib/gerken--et-al--2023.bib}
\nocite{2021arXiv210413478B}
\bibliographystyle{authordate1}

\appendix


\section{Constructing Families of Measures}
\label{sec:famMeas}

In this Appendix \ref{sec:famMeas} we show
how one can construct families of measures
as used in this paper
in a natural way
leaving as little to choice as possible.
We start by providing the following notion.

\begin{definition}
  \label{definition:freeGspace}
  We say that a $G$-space $X$ is \emph{free},
  if there is a topological space $M$
  and a $G$-equivariant homeomorphism
  \begin{equation*}
    X \cong M \times G.
  \end{equation*}
\end{definition}

\begin{remark}
  Note that the $G$-action associated to any free $G$-space
  is necessarily fixed-point free.
  However,
  there are non-free $G$-spaces with fixed-point free actions
  as for example $\R$ as a $\Z$-space under addition.
\end{remark}

\subsection{Families of Haar Measures}
\label{sec:famHaarMeas}

In order to provide a natural construction
of a family of measures
\[\{\mu_b \colon \mathcal{B}(G) \rightarrow [0, \infty]\}_{b \in B},\]
we impose the following.

\begin{assumption}
  \label{assumption:HaarMeas}
  We assume that $G$ is locally compact
  and that for any ${b \in B}$ we have
  ${G_b \subseteq N}$ for the corresponding stabilizer,
  where ${N \coloneqq \ker \Delta = \Delta^{-1}(1)}$
  is the kernel of the modular function
  ${\Delta \colon G \rightarrow (0, \infty)}$,
  see for example \mbox{\citep[Section 3]{2020arXiv200610956T}}.
  Moreover,
  we assume that
  ${N \backslash B}$
  has a ${G / N}$-invariant
  locally finite partition of unity
  such that each induced open is a free ${G / N}$-space
  in the sense of \cref{definition:freeGspace}.
\end{assumption}

\begin{examples}
  We provide the following two cases,
  where the \cref{assumption:HaarMeas}
  is satisfied.
  \begin{enumerate}[(i)]
  \item
    If $G$ is unimodular,
    then ${N = \Delta^{-1}(1) = G}$ and
    \begin{equation*}
      N \backslash B = G \backslash B \cong
      G \backslash B \times G / G.
    \end{equation*}
    So we may choose the single constant function
    \begin{equation*}
      \varphi \colon G \backslash B \rightarrow [0, 1], \,
      G.b \mapsto 1
    \end{equation*}
    for a partition of unity $\{\varphi\}$.
  \item
    If $G$ acts transitively on $B$,
    then ${G / N}$ acts freely and transitively on ${N \backslash B}$.
    So any choice of an $N$-orbit $N.b$ yields an isomorphism
    \begin{equation*}
      G / N \rightarrow N \backslash B, \,
      gN = Ng \mapsto Ng.b.
    \end{equation*}
    Thus,
    we may again choose the constant function
    \begin{equation*}
      \varphi \colon N \backslash B \rightarrow [0, 1], \,
      N.b \mapsto 1
    \end{equation*}
    for a partition of unity $\{\varphi\}$.
  \end{enumerate}
\end{examples}

\begin{lemma}
  Under the \cref{assumption:HaarMeas}
  there is a continuous function
  ${\lambda \colon B \rightarrow (0, \infty)}$
  such that
  ${\lambda(g.b) = \Delta(g) \lambda(b)}$
  for all ${g \in G}$ and ${b \in B}$.
\end{lemma}

\begin{proof}
  \sloppy
  It suffices to provide a continuous function
  ${\bar{\lambda} \colon N \backslash B \rightarrow \R}$
  such that
  \begin{equation*}
    \bar{\lambda}(Ng.b) =
    \log \Delta(g) + \bar{\lambda}(N.b)
  \end{equation*}
  for all ${g \in G}$ and ${b \in B}$
  as is easily seen considering the diagram
  \begin{equation*}
    \begin{tikzcd}[ampersand replacement=\&,row sep=5ex, column sep=8ex]
      B
      \arrow[r, "\lambda"]
      \arrow[d, two heads]
      \&
      (0, \infty)
      \\
      N \backslash B
      \arrow[r, "\bar{\lambda}"']
      \&
      \R. \!
      \arrow[u, "\exp"']
    \end{tikzcd}
  \end{equation*}
  To this end,
  let
  \begin{equation*}
    \{\varphi_i \colon N \backslash B \rightarrow [0, 1]\}_{i \in I}
  \end{equation*}
  be a locally finite
  ${G / N}$-invariant partition of unity
  such that
  ${\varphi_i^{-1}(0, 1]}$
  is a free ${G / N}$-space for all ${i \in I}$.
  Moreover,
  we choose ${G / N}$-equivariant homeomorphisms
  \begin{equation}
    \label{eq:presentationPartUnity_lambda}
    \varphi_i^{-1}(0, 1] \cong U_i \times G / N
  \end{equation}
  for ${i \in I}$.
  By precomposing each of the functions
  \begin{equation*}
    U_i \times G / N \rightarrow \R, \,
    (p, gN) \mapsto \log \Delta(g)
    \quad
    \text{where ${i \in I}$}
  \end{equation*}
  with the corresponding homeomorphism
  from \eqref{eq:presentationPartUnity_lambda}
  we obtain functions
  ${\bar{\lambda}_i \colon
    \varphi_i^{-1}(0, 1] \rightarrow \R}$
  such that
  \begin{equation}
    \label{eq:localCoeffAdditive}
    \bar{\lambda}_i (gN.b) =
    \log \Delta(g) + \bar{\lambda}_i (N.b)
  \end{equation}
  for all ${i \in I}$, ${g \in G}$,
  and ${b \in \bigcup_{i \in I} \varphi_i^{-1}(0, 1]}$.
  Then we define
  \begin{equation*}
    \bar{\lambda} \colon
    N \backslash B \rightarrow \R, \,
    p \mapsto
    \sum_{i \in I} \varphi_i(p) \bar{\lambda}_i(p).
  \end{equation*}
  As ${\varphi_i(p) = 0}$ for ${i \in I}$ and
  ${p \in (N \backslash B) \setminus \varphi_i^{-1}(0, 1]}$
  and as the partition of unity $\{\varphi_i\}_{i \in I}$
  is locally finite,
  the function ${\bar{\lambda}}$ is well-defined and continuous.

  Now let ${g \in G}$ and ${b \in B}$.
  Then we have
  \begin{align*}
    \bar{\lambda}(Ng.b)
    &=
      \bar{\lambda}(gN.b)
    \\[1ex]
    &=
      \sum_{i \in I} \varphi_i (gN.b) \bar{\lambda}_i(gN.b)
    \\[1ex]
    &=
      \sum_{i \in I} \varphi_i (N.b) \bar{\lambda}_i(gN.b)
    \\[1ex]
    &=
      \sum_{\mathclap{\substack{i \in I \\ \varphi_i(N.b) \neq 0}}}
      \varphi_i (N.b) \bar{\lambda}_i(gN.b)
    \\[1ex]
    &\!\!\stackrel{\eqref{eq:localCoeffAdditive}}{=}
      \sum_{\mathclap{\substack{i \in I \\ \varphi_i(N.b) \neq 0}}}
      \varphi_i (N.b)
      \big(
      \log \Delta(g) + \bar{\lambda}_i (N.b)
      \big) 
    \\[1ex]
    &=
      \log \Delta(g) + 
      \sum_{\mathclap{\substack{i \in I \\ \varphi_i(N.b) \neq 0}}}
      \varphi_i (N.b)
      \bar{\lambda}_i (N.b)
    \\[1ex]
    &=
      \log \Delta(g) + 
      \sum_{i \in I}
      \varphi_i (N.b)
      \bar{\lambda}_i (N.b)
    \\[1ex]
    &=
      \log \Delta(g) +
      \bar{\lambda} (N.b).
  \end{align*}
  Here the third equality follows from the ${G / N}$-invariance
  of the partition of unity
  $\{\varphi_i\}_{i \in I}$
  and the sixth equality from
  \begin{equation*}
    \sum_{\mathclap{\substack{i \in I \\ \varphi_i(N.b) \neq 0}}}
    \varphi_i (N.b)
    =
    1.
    \qedhere
  \end{equation*} 
\end{proof}

\begin{theorem}
  \label{thm:existenceFamHaarMeas}  
  Under the \cref{assumption:HaarMeas}
  there is a continuous family
  \[\{\mu_b \colon \mathcal{B}(G) \rightarrow [0, \infty]\}_{b \in B}\]
  of Haar measures on $G$ such that
  \begin{equation*}
    \mu_{g.b} = c_{g*} \mu_b
    \tag{\ref{eq:muInv} revisited}
  \end{equation*}
  for all ${g \in G}$ and ${b \in B}$,
  where
  ${c_{g*} \mu_b}$ is the pushforward measure
  of $\mu_b$ along the conjugation
  \begin{equation*}
    c_g \colon G \rightarrow G, \,
    h \mapsto g h g^{-1}.
  \end{equation*}
\end{theorem}

\begin{proof}
  Let
  ${\mu_0 \colon \mathcal{B}(G) \rightarrow [0, \infty]}$
  be some Haar measure on $G$
  and let
  ${\lambda \colon B \rightarrow (0, \infty)}$
  be as in the previous lemma.
  We set ${\mu_b \coloneqq \lambda_b \mu_0}$.
  Now let ${A \subseteq G}$ be some Borel set.
  Then we have
  \begin{align*}
    \mu_{g.b} (A)
    &=
      \lambda(g.b) \mu_0(A)
    \\
    &=
      \Delta(g) \lambda(b) \mu_0(A)
    \\
    &=
      \Delta(g) \mu_b(A)
    \\
    &=
      \mu_b(Ag)
    \\
    &=
      \mu_b\big(g^{-1} Ag\big)
    \\
    &=
      \mu_b\big(c_g^{-1}(A)\big)
    \\
    &=
      c_{g*} \mu_b(A).
  \end{align*}
  Finally,
  let ${f \colon G \rightarrow \R}$
  be continuous and compactly supported.
  Then we have
  \begin{equation*}
    \int_G f(h) d\mu_b(h) =
    \int_G \lambda(b) f(h) d\mu_0(h) =
    \lambda(b) \int_G f(h) d\mu_0(h).
  \end{equation*}
  Thus,
  the function
  \begin{equation*}
    B \rightarrow \R, \,
    b \mapsto \int_G f(h) d \mu_b(h)
  \end{equation*}
  is continuous.
\end{proof}

\subsection{Families of Group Invariant Measures}
\label{sec:famInvMeas}

In addition to a family of Haar measures $\{\mu_b\}_{b \in B}$
as in the previous \cref{sec:famHaarMeas},
we now describe the construction
of compatible $G$-invariant measures
on the orbits of the action ${G \curvearrowright B}$
as well as Haar measures on the stabilizers.

\begin{assumption}
  \label{assumption:GInvMeas}
  We impose the \cref{assumption:HaarMeas} and moreover,
  we assume
  \begin{itemize}
  \item
    we have a non-vanishing
    compactly supported continuous function
    \[\psi_0 \colon G \rightarrow [0, \infty)\]
    invariant under conjugation by elements of $N$
  \item
    and for any ${b \in B}$ the stabilizer $G_b$ is unimodular
    and the map
    \begin{equation*}
      \_.b \colon
      G \rightarrow G.b \subseteq B, \,
      g \mapsto g.b
    \end{equation*}
    is a quotient map.
  \end{itemize}
\end{assumption}

\begin{lemma}
  \label{thm:psi}
  \sloppy
  Under the \cref{assumption:GInvMeas}
  there is a continuous function
  \begin{equation*}
    \psi \colon G \times B \rightarrow [0, \infty)
  \end{equation*}
  such that
  \begin{equation}
    \label{eq:psiConstraint}
    \psi(h, b) = \psi\big(ghg^{-1}, g.b\big)
  \end{equation}
  for all ${g, h \in G}$ and ${b \in B}$
  and each partially applied function
  ${\psi(-, b) \colon G \rightarrow [0, \infty)}$
  (where ${b \in B}$)
  is a convex combination of functions of the form
  \begin{equation*}
    G \rightarrow [0, \infty), \,
    h \mapsto \psi_0\big(ghg^{-1}\big)
  \end{equation*}
  with ${g \in G}$.
\end{lemma}

\begin{proof}
  \sloppy
  It suffices to construct a continuous function
  \begin{equation*}
    \bar{\psi} \colon G \times (N \backslash B) \rightarrow [0, \infty)
  \end{equation*}
  such that
  \begin{equation*}
    \bar{\psi}(h, N.b) = \bar{\psi}\big(ghg^{-1}, Ng.b\big)
  \end{equation*}
  for all ${g, h \in G}$ and ${b \in B}$
  and each partially applied function
  ${\bar{\psi}(-, N.b) \colon G \rightarrow [0, \infty)}$
  (where ${b \in B}$)
  is a convex combination of functions of the form
  \begin{equation*}
    G \rightarrow [0, \infty), \,
    h \mapsto \psi_0\big(ghg^{-1}\big)
  \end{equation*}
  with ${g \in G}$.
  To this end,
  let
  \begin{equation*}
    \{\varphi_i \colon N \backslash B \rightarrow [0, 1]\}_{i \in I}
  \end{equation*}
  be a locally finite
  ${G / N}$-invariant partition of unity
  such that
  ${\varphi_i^{-1}(0, 1]}$
  is a free ${G / N}$-space for all ${i \in I}$.
  Moreover,
  we choose ${G / N}$-equivariant homeomorphisms
  \begin{equation}
    \label{eq:presentationPartUnity_psi}
    \varphi_i^{-1}(0, 1] \cong U_i \times G / N
  \end{equation}
  for ${i \in I}$.
  By combining the homeomorphisms
  \eqref{eq:presentationPartUnity_psi}
  with the functions
  \begin{equation*}
    G \times U_i \times G /N \rightarrow [0, \infty), \,
    (h, p, gN) \mapsto \psi_0\big(ghg^{-1}\big)
  \end{equation*}
  we obtain functions
  \begin{equation*}
    \bar{\psi}_i \colon
    G \times \varphi_i^{-1}(0, 1] \rightarrow [0, \infty)
  \end{equation*}
  such that
  \begin{equation*}
    \bar{\psi}_i(h, N.b) =
    \psi_0\big(ghg^{-1}\big)
    \quad
    \text{for some ${g \in G}$}
  \end{equation*}
  and for all ${i \in I}$
  and ${b \in \bigcup_{i \in I} \varphi_i^{-1}(0, 1]}$
  and such that
  \begin{equation}
    \label{eq:psiQuotConjugation}
    \bar{\psi}_i (h, N.b) = \bar{\psi}_i \big(ghg^{-1}, Ng.b\big)
  \end{equation}
  for all ${i \in I}$, ${g, h \in G}$,
  and ${b \in \bigcup_{i \in I} \varphi_i^{-1}(0, 1]}$.
  Then we define
  \begin{equation*}
    \bar{\psi} \colon
    G \times N \backslash B \rightarrow [0, \infty), \,
    (h, N.b) \mapsto \sum_{i \in I} \varphi_i(N.b) \bar{\psi}_i(h, N.b).
  \end{equation*}
  As ${\varphi_i(p) = 0}$ for ${i \in I}$ and
  ${p \in (N \backslash B) \setminus \varphi_i^{-1}(0, 1]}$
  and as the partition of unity $\{\varphi_i\}_{i \in I}$
  is locally finite,
  the function $\bar{\psi}$ is well-defined, continuous,
  and each partially applied function ${\bar{\psi}(-, N.b)}$
  (where ${b \in B}$)
  is a convex combination of functions of the form
  \begin{equation*}
    G \rightarrow [0, \infty), \,
    h \mapsto \psi_0\big(ghg^{-1}\big)
  \end{equation*}
  with ${g \in G}$.
  Now let ${g, h \in G}$ and ${b \in B}$.
  Then we have
  \begin{align*}
    \bar{\psi}(h, N.b)
    &=
      \sum_{i \in I} \varphi_i(N.b) \bar{\psi}_i(h, N.b)
    \\
    &=
      \sum_{i \in I} \varphi_i(gN.b) \bar{\psi}_i\big(ghg^{-1}, Ng.b\big)
    \\
    &=
      \sum_{i \in I} \varphi_i(Ng.b) \bar{\psi}_i\big(ghg^{-1}, Ng.b\big)
    \\
    &=
      \bar{\psi}\big(ghg^{-1}, Ng.b\big)
  \end{align*}
  Here the second equality follows from
  \eqref{eq:psiQuotConjugation}
  and ${G/N}$-invariance of the partition of unity
  $\{\varphi_i\}_{i \in I}$.
\end{proof}

\begin{theorem}
  \label{thm:compatFamMeas}
  Under the \cref{assumption:GInvMeas}
  there is a family
  \begin{equation*}
    \{
    \bar{\mu}_b \colon
    \mathcal{B}(G.b) \rightarrow [0, \infty]
    \}_{b \in B}
  \end{equation*}
  of $G$-invariant Radon measures,
  there are families
  \begin{equation*}
    \{
    \mu_b \colon
    \mathcal{B}(G) \rightarrow [0, \infty]
    \}_{b \in B}
    \quad \text{and} \quad
    \{
    \nu_b \colon
    \mathcal{B}(G_b) \rightarrow [0, \infty]
    \}_{b \in B}
  \end{equation*}
  of Haar measures,
  and there is a function
  ${\psi \colon G \times B \rightarrow [0, \infty)}$
  as in the previous \cref{thm:psi}
  such that
  \begin{equation*}
    \int_G
    \psi(h, b) d\mu_b(h)
    =
    1
    =
    \int_{G_b}
    \psi(h, b) d\nu_b(h)
  \end{equation*}
  for all ${b \in B}$,
  such that
  \begin{equation*}
    \int_G
    f(h) d\mu_b(h)
    =
    \int_{G.b}
    \int_{G_b} f(kh) d\nu_b(h)
    d\bar{\mu}_b(k.b)
    \tag{\ref{eq:fubini} revisited}
  \end{equation*}
  for all ${b \in B}$
  and compactly supported continuous
  ${f \colon G \rightarrow \R}$,
  and such that
  \begin{align}
    \label{eq:baseMeasPushforward}
    \bar{\mu}_{g.b}
    &=
      (g.\_)_* \bar{\mu}_{b},
    \\
    \notag
    \mu_{g.b}
    &=
      c_{g*} \mu_{b},
    \\
    \notag
    \text{and}
    \quad
    \nu_{g.b}
    &=
      c_{g*} \nu_{b}    
  \end{align}
  for all ${g \in G}$ and ${b \in B}$,
  where $c_g$ denotes the corresponding of the two
  vertical conjugation maps in
  \begin{equation*}
    \begin{tikzcd}[ampersand replacement=\&,row sep=5ex, column sep=7ex]
      G_b
      \arrow[r, hook]
      \arrow[d, "c_g"']
      \&
      G
      \arrow[d, "c_g"]
      \&[-7ex]
      h
      \arrow[d, mapsto]
      \\
      G_{g.b}
      \arrow[r, hook]
      \&
      G
      \&
      ghg^{-1}.
    \end{tikzcd}
  \end{equation*}
  Moreover,
  the family ${\{\mu_b\}_{b \in B}}$ is continuous
  and the partially applied function
  ${\psi(-, b) \colon G \rightarrow \R}$
  is compactly supported for each ${b \in B}$.
\end{theorem}

\begin{remark}
  We note that as $\bar{\mu}_b$ is stated to be $G$-invariant
  for any ${b \in B}$ in \cref{thm:compatFamMeas},
  we could have stated equation \eqref{eq:baseMeasPushforward} as
  ${\bar{\mu}_{g.b}
    =
    \bar{\mu}_{b}}$
  instead.
  However,
  \eqref{eq:baseMeasPushforward}
  is the equation that we will need
  and it is also more elementary to prove,
  i.e. without directly invoking $G$-invariance.
\end{remark}

\begin{proof}
  Let
  ${\psi \colon G \times B \rightarrow [0, \infty)}$
  be as in the previous \cref{thm:psi}
  and suppose we have some ${b \in B}$.
  Let
  \begin{equation*}
    \mu_b \colon \mathcal{B}(G) \rightarrow [0, \infty]
    \quad \text{and} \quad
    \nu_b \colon \mathcal{B}(G_b) \rightarrow [0, \infty]
  \end{equation*}
  be the unique Haar measures such that
  \begin{equation*}
    \int_G \psi(h, b) d\mu_b(h)
    = 1 =
    \int_{G_b} \psi(h, b) d\nu_b(h).
  \end{equation*}
  It is well known there is a unique $G$-invariant Radon measure
  \begin{equation*}
    \bar{\mu}_b \colon
    \mathcal{B}(G.b) \rightarrow
    [0, \infty)
  \end{equation*}
  such that
  \begin{equation*}
    \int_G
    f(h) d\mu_b(h)
    =
    \int_{G.b}
    \int_{G_b} f(kh) d\nu_b(h)
    d\bar{\mu}_b(k.b)
    \tag{\ref{eq:fubini} revisited}
  \end{equation*}
  for any compactly supported continuous function
  ${f \colon G \rightarrow \R}$,
  see for example \mbox{\citep[Theorem 4.2]{2020arXiv200610956T}}.
  Now let ${g \in G}$ and ${b \in B}$.
  In order to show
  ${\nu_{g.b} = c_{g*} \nu_b}$
  it suffices to test ${c_{g*} \nu_b}$
  on the partially applied function
  ${\psi(-, g.b)|_{G_b} \colon
    G_{g.b} \rightarrow
    [0, \infty)}$
  as here
  \begin{equation*}
    \int_{G_{g.b}} \psi(h, g.b) d c_{g*}\nu_b(h) =
    \int_{G_b} \psi\big(ghg^{-1}, g.b\big) d \nu_b(h)
    =
    \int_{G_b} \psi(h, b) d \nu_b(h)
    =
    1.
  \end{equation*}
  Completely analogously
  we obtain
  ${\mu_{g.b} = c_{g*} \mu_b}$.
  Now let
  ${f \colon G \rightarrow \R}$
  be continuous and compactly supported.
  Then we have
  \begin{align*}
    \int_{G.b}
    \int_{G_{g.b}}
    f(kh)
    d \nu_{g.b}(h)
    d (g.\_)_* \bar{\mu}_b (kg.b)
    \!\!\!\!\!\!\!\!\!\!\!\!\!\!\!\!\!\!\!\!\!\!\!\!\!\!\!\!\!\!
    \!\!\!\!\!\!\!\!\!\!\!\!\!\!\!\!\!\!\!\!\!\!\!\!\!\!\!\!\!\!
    \\
    &=
      \int_{G.b}
      \int_{G_{g.b}}
      f(kh)
      d c_{g*} \nu_{b}(h)
      d (g.\_)_* \bar{\mu}_b (kg.b)
    \\
    &=
      \int_{G.b}
      \int_{G_{b}}
      f\big(kghg^{-1}\big)
      d \nu_{b}(h)
      d (g.\_)_* \bar{\mu}_b (kg.b)
    \\
    &=
      \int_{G.b}
      \int_{G_{b}}
      f\big(khg^{-1}\big)
      d \nu_{b}(h)
      d (g.\_)_* \bar{\mu}_b (k.b)
    \\
    &=
      \int_{G.b}
      \int_{G_{b}}
      f\big(gkhg^{-1}\big)
      d \nu_{b}(h)
      d \bar{\mu}_b (k.b)
    \\
    &\stackrel{\eqref{eq:fubini}}{=}
      \int_{G}
      f\big(ghg^{-1}\big)
      d \mu_b (h)
    \\
    &=
      \int_{G}
      f(h)
      d c_{g*} \mu_b (h)
    \\
    &=
      \int_{G}
      f(h)
      \mu_{g.b} (h).
  \end{align*}
  As we defined
  ${\bar{\mu}_{g.b}}$
  as the unique $G$-invariant Radon measure
  satisfying \eqref{eq:fubini}
  with $g.b$ substituted for $b$,
  we get ${(g.\_)_* \bar{\mu}_b = \bar{\mu}_{g.b}}$.
  Moreover,
  the partially applied function
  ${\psi(-, b) \colon G \rightarrow [0, \infty)}$
  is compactly supported
  as it is a convex combination of functions of the form
  \begin{equation*}
    G \rightarrow [0, \infty), \,
    h \mapsto \psi_0\big(ghg^{-1}\big)
  \end{equation*}
  with ${g \in G}$,
  and its integral
  ${\int_G
    \psi(h, b)
    d\mu_0(h)}$
  is non-zero for the same reason.
  
  Finally,
  let ${f \colon G \rightarrow \R}$
  be continuous and compactly supported as before
  and let
  ${\mu_0 \colon \mathcal{B}(G) \rightarrow [0, \infty]}$
  be some Haar measure on $G$.
  Then we have
  \begin{align*}
    \int_G
    \psi(h, b)
    d\mu_0(h)
    \int_G
    f(h)
    d\mu_b(h)
    &=
      \int_G
      \psi(h, b)
      d\mu_b(h)
      \int_G
      f(h)
      d\mu_0(h)
    \\
    &=
      \int_G
      f(h)
      d\mu_0(h)
  \end{align*}
  for all ${b \in B}$.
  As the integral
  ${\int_G
    \psi(h, b)
    d\mu_0(h)}$
  is non-zero and continuous in ${b \in B}$,
  the integral
  ${\int_G
    f(h)
    d\mu_b(h)}$
  is continuous in ${b \in B}$ as well.
\end{proof}

\section{Vanishing Dual Sections}
\label{sec:vanishingDualSections}

Let $B$ be a paracompact space,
let ${E \rightarrow B}$ be a real vector bundle over $B$,
and let
${\mu \colon \mathcal{B}(B) \rightarrow [0, \infty]}$
be some strictly positive locally finite Borel measure on $B$.
Moreover,
let ${\sigma \in \Gamma_c(E^*)}$ be a compactly supported continuous section
of the dual bundle ${E^* \rightarrow B}$ associated to $E$.
(We use $c$ as a subscript to $\Gamma$
to denote compactly supported sections.)

\begin{lemma}
  \label{thm:vanishingDualSections}
  If we have
  \begin{equation*}
    \int_B \sigma(f(b)) d\mu(b) = 0
  \end{equation*}
  for all compactly supported continuous sections
  ${f \in \Gamma_c(E)}$,
  then ${\sigma = 0}$.
\end{lemma}

\begin{proof}
  By \mbox{\citep[Proposition 1.2]{hatcherKTheory}}
  there is an inner product
  \begin{equation*}
    \langle \_, \_ \rangle \colon
    E \oplus E \rightarrow \R
  \end{equation*}
  on $E$.
  Let ${f \coloneqq \sigma^{\sharp} \in \Gamma_c(E)}$
  be the section corresponding to $\sigma$ under the musical isomorphism
  induced by the inner product ${\langle \_, \_ \rangle}$.
  Then we have
  \begin{equation*}
    \sigma(f(b)) = \langle f(b), f(b) \rangle \geq 0
  \end{equation*}
  for all ${b \in B}$ and moreover,
  \begin{equation*}
    \int_B \langle f(b), f(b) \rangle d\mu(b) =
    \int_B \sigma(f(b)) d\mu(b) = 0,
  \end{equation*}
  hence ${f(b) = 0}$ for all ${b \in B}$.
  Applying the musical isomorphism once more we obtain
  ${\sigma = f^{\flat} = 0}$.
\end{proof}

\section{Locally Finite Covers and Compact Subsets}

Let $X$ be a topological space,
let ${\{U_i \subseteq X\}_{i \in I}}$
be a locally finite cover of $X$,
and let ${C \subseteq X}$
be a compact subset.
(In our application of the following \cref{thm:locallyFiniteCoverCompact},
each subset ${U_i \subseteq X}$ is open in $X$.
In order to prove \cref{thm:locallyFiniteCoverCompact},
this assumption is not needed however.)

\begin{lemma}
  \label{thm:locallyFiniteCoverCompact}
  The set
  ${\{i \in I \mid U_i \cap X \neq \emptyset\}}$
  is finite.
\end{lemma}

\begin{proof}
  For each ${x \in X}$
  we choose an 
  open neighbourhood ${V_x \subseteq X}$ of $x$
  such that ${U_i \cap V_x \neq \emptyset}$
  for only a finite number of ${i \in I}$.
  Then we have the open cover
  ${\{V_x \cap C\}_{x \in X}}$
  of ${C}$.
  Moreover,
  as ${C}$
  is compact,
  there is a finite subset ${J \subseteq X}$
  with
  \[C \subseteq \bigcup_{j \in J} V_j.\]

  Now let ${i \in I}$.
  Then there is some
  \begin{equation*}
    y \in U_i \cap C \subseteq \bigcup_{j \in J} V_j.
  \end{equation*}
  Thus,
  there is some ${j_i \in J}$
  with ${y \in V_{j_i}}$ and hence
  \begin{equation}
    \label{eq:coverIntersection}
    U_i \cap V_{j_i} \neq \emptyset.
  \end{equation}

  Now let
  ${P \coloneqq
    \{(i, j) \in I \times J \mid U_i \cap V_j \neq \emptyset\}.}$
  Then we have the map
  \begin{equation}
    \label{eq:coverInjection}
    I \rightarrow P,\,
    i \mapsto (i, j_i),
  \end{equation}
  which is well-defined by \eqref{eq:coverIntersection}
  and injective by construction.
  Moreover,
  the set $P$
  is finite by construction
  of the family $\{V_x\}_{x \in X}$
  and the finiteness of $J$.
  In conjunction with the injectivity of the map
  \eqref{eq:coverInjection},
  the set $I$ is finite as well.  
\end{proof}


\end{document}